\documentclass{amsart}
\usepackage{amssymb}
\usepackage{amsfonts}

\newtheorem{theorem}{Theorem}
\theoremstyle{plain}

\newtheorem{corollary}{Corollary}

\newtheorem{definition}{Definition}

\newtheorem{lemma}{Lemma}

\newtheorem{proposition}{Proposition}
\newtheorem{remark}{Remark}

\numberwithin{equation}{section}
\input{tcilatex}

\begin{document}
\title[ $\ast $-SH Domains]{On $\ast $-Semi Homogeneous Domains}
\author{D.D. Anderson}
\address{Department of Mathematics, The University of Iowa, Iowa City, IA
52242}
\email{\"{y} }
\author{Muhammad Zafrullah}
\address{Department of Mathematics, Idaho State University, Pocatello, 83209
ID}
\email{mzafrullah@usa.net}
\urladdr{http://www.lohar.com}
\date{March 30}
\thanks{ We are thankful to Professor S. El Baghdadi for numerous
corrections. Advice from Professors D. Dobbs and G. Bergman was also useful.}

\begin{abstract}
Let $\ast $ be a finite character star operation defined on an integral
domain $D.$ Call a nonzero $\ast $-ideal $I$ of finite type a $\ast $%
-homogeneous ($\ast $-homog) ideal, if $I\subsetneq D$ and $(J+K)^{\ast
}\neq D$ for every pair $D\supsetneq J,K\supseteq I$ of proper $\ast $%
-ideals of finite type$.$ Call an integral domain $D$ a $\ast $-Semi
Homogeneous Domain ($\ast $-SHD) if every proper principal ideal $xD$ of $D$
is expressible as a $\ast $-product of finitely many $\ast $-homog ideals.
We show that a $\ast $-SHD contains a family $\mathcal{F}$ of prime ideals
such that (a) $D=\cap _{P\in \mathcal{F}}D_{P},$ a locally finite
intersection and (b) no two members of $\mathcal{F}$ contain a common non
zero prime ideal. The $\ast $-SHDs include h-local domains, independent
rings of Krull type, Krull domains, UFDs etc. We show also that we can
modify the definition of the $\ast $-homog ideals to get a theory of each
special case of a $\ast $-SH domain.
\end{abstract}

\maketitle
\dedicatory{Dedicated to the memory of Professor P.M. Cohn}

\section{ Introduction}

Let $\ast $ be a finite character star operation defined on an integral
domain $D.$(We will, in the following, introduce terminology necessary
\bigskip for reading this article.) Call a nonzero $\ast $-ideal of finite
type a $\ast $-homogeneous ($\ast $-homog) ideal, if $I\subsetneq D$ and $%
(J+K)^{\ast }\neq D$ for every pair $D\supsetneq J,K\supseteq I$ of proper $%
\ast $-ideals of finite type$.$ To fix the ideas the simplest example of a $%
\ast $-homog ideal is an ideal generated by some positive power of a
principal prime. The initial aim of this article is to show that if a $\ast $%
-ideal $A$ is expressible as a $\ast $-product of finitely many $\ast $%
-homog ideals, then $A$ is uniquely expressible as a $\ast $-product of
finitely many mutually $\ast $-comaximal $\ast $-homog ideals. Call an
integral domain $D$ a $\ast $-Semi Homogeneous Domain($\ast $-SH domain or $%
\ast $-SHD) if for every nonzero non unit $x$ of $D$ the ideal $xD$ is
expressible as a $\ast $-product of finitely many $\ast $-homog ideals. The
purpose of this paper is then to show that a $\ast $-SHD is a $\mathcal{F}$%
-IFC domain of \cite{AZ2} that is a $\ast $-SHD contains a family of prime
ideals $\mathcal{F}$ such that (a) $D=\cap _{P\in \mathcal{F}}D_{P}$ and the
intersection is locally finite and (b) no two members of $\mathcal{F}$
contain a nonzero prime ideal. It turns out that the $\ast $-SHDs contain as
special cases the h-local domains of Matlis \cite{Mat} and an important
paper about them by \cite{Olb}), independent rings of Krull type of Griffin 
\cite{Gr1}, Krull domains, weakly Krull domains see \cite{AMZ}, UFD's etc.
What is special with our approach is that for each kind of domains we can
modify the definition of the $\ast $-homog ideals to give a theory of that
kind of domains. But before we explain that, let's bring in the above
promised introduction to the terminology.

Since this work is steeped in and dependent upon the unifying quality of
star operations, it seems pertinent to give the reader a working knowledge
of some of the notions involved. But before we delve into that, let's
indicate the quite a few theories that we can generate, if we choose to
ignore all mention of star-operations.

What we have developed here is a kind of \textquotedblleft theory
schema\textquotedblright . Let us explain: You may want to give
star-operations a wide berth, for some reason (s ), but you want to have
some idea of what we have developed. Here's what we can say: If you ignore
the \textquotedblleft *\textquotedblright , completely and everywhere, you
would get the main theory of SH domains that characterizes h-local domains
of Matlis and if you change the definition of a homog ideal a little you
would get the theory of unique factorization characterizing 1-dimensional
h-local domains. Another change of definition of homog ideals gives you
h-local Prufer domains. Proceeding this way and changing definitions of
homog ideals judiciously, or as presented in this paper, you can get to the
theories of Dedekind domains and PID's. If you have time and inclination, do
look into the paper that way. Finally if you know of the easy to define star
operation, called the $t$-operation, you would get another round of theories
leading to Krull domains and various kinds of Krull domains and eventually
to UFDs.

Let $D$ be an integral domain with quotient field $K.$ Let $F(D)$ (resp., $%
f(D)$) be the set of nonzero fractional ideals (resp., nonzero finitely
generated fractional ideals) of $D$. A star operation $\ast $ on $D$ is a
closure operation on $F(D)$ that satisfies the following properties for
every $I,J\in F(D)$ and $0\neq x\in K$:

(i) $(x)^{\ast} = (x)$ and $(xI)^{\ast} = xI^{\ast}$,

(ii) $I \subseteq I^{\ast}$, and $I^{\ast} \subseteq J^{\ast}$ whenever $I
\subseteq J$, and

(iii) $(I^{\ast})^{\ast} = I^{\ast}$.

\vspace{0.1cm} \noindent Now, an $I\in F(D)$ is a $\ast $-ideal if $I^{\ast
}=I,$ so a principal ideal is a $\ast $-ideal for every star operation $\ast
.$ Moreover $I\in F(D)$ is called a $\ast $-ideal of finite type if $%
I=J^{\ast }$ for some $J\in f(D)$. To each star operation $\ast $ we can
associate a star operation $\ast _{s}$ defined by $I^{\ast _{s}}=\bigcup
\{\,J^{\ast }\mid J\subseteq I$ and $J\in f(D)\,\}.$ A star operation $\ast $
is said to be of finite type, if $I^{\ast }=I^{\ast _{s}}$ for all $I\in
F(D).$ Indeed for each star operation $\ast ,$ $\ast _{s}$ is of finite
character. Thus if $\ast $ is of finite character $I\in F(D)$ is a $\ast $%
-ideal if and only if for each finitely generated subideal $J$ of $I$ we
have $J^{\ast }\subseteq I.$ For $I\in F(D)$, let $I_{d}=I$, $%
I^{-1}=(D:_{K}I)=\{\,x\in K\mid xI\subseteq D\,\}$, $I_{v}=(I^{-1})^{-1}$, $%
I_{t}=\bigcup \{\,J_{v}\mid J\subseteq I$ and $J\in f(D)\,\}$, and $%
I_{w}=\{\,x\in K\mid xJ\subseteq I$ for some $J\in f(D)$ with $J_{v}=\,D\,\}$%
. The functions defined by $I\mapsto I_{d}$ , $I\mapsto I_{v}$ , $I\mapsto
I_{t}$, and $I\mapsto I_{w}$ are all examples of star operations A $v$-ideal
is sometimes also called a divisorial ideal. Given two star operations $\ast
_{1},\ast _{2}$ on $D$, we say that $\ast _{1}\leq \ast _{2}$ if $I^{\ast
_{1}}\subseteq I^{\ast _{2}}$ for every $I\in F(D)$. Note that $\ast
_{1}\leq \ast _{2}$ if and only if $(I^{\ast _{1}})^{\ast _{2}}=(I^{\ast
_{2}})^{\ast _{1}}=I^{\ast _{2}}$ for every $I\in F(D)$. The $d$-operation, $%
t$-operation, and $w$-operation all have finite character, $d\leq \rho \leq
v $ for every star operation $\rho $, and $\rho \leq t$ for every star
operation $\rho $ of finite character. We will often use the facts that (a)
for every star operation $\ast $ and $I,J\in F(D),~(IJ)^{\ast }=(IJ^{\ast
})^{\ast }=(I^{\ast }J^{\ast })^{\ast },$ (the $\ast $-product), (b) $%
(I+J)^{\ast }=(I+J^{\ast })^{\ast }=(I^{\ast }+J^{\ast })^{\ast }$ (the $%
\ast $-sum) and (c) $I_{v}=I_{t}$ for every $I\in f(D)$. An $I\in F(D)$ is
said to be $\ast $-invertible, if $(II^{-1})^{\ast }=D$. If $I$ is $\ast $%
-invertible for $\ast $ of finite character, then both $I^{\ast }$ and $%
I^{-1}$ are $\ast $-ideals of finite type. An integral domain $D$ is called
a Prufer $\ast $-Multiplication Domain (P$\ast $MD), for a general star
operation $\ast $, if $A$ is $\ast _{s}$-invertible for every $A\in f(D).$
Now let $D$ be a P$\ast $MD. Because in a P$\ast $MD $D,$ $F^{\ast }=F_{v}$
for each $F\in f(D),$ we have $A^{\ast _{s}}=A_{t}$ for each $A\in F(D).$
(When $\ast $ is of finite character, $\ast =\ast _{s}$ and so in such a P$%
\ast $MD $D,$ we have $A^{\ast }=A_{t}$ for each $A\in F(D)$ and so $\ast =t$%
. Moreover, in a P$d$MD $d=t$, making a P$d$MD a Prufer domain.) A P$v$MD is
often written as PVMD. A reader in need of more introduction may consult 
\cite{z00} or \cite[Sections 32 and 34]{gilmer}.

For a star operation $\ast $, a maximal $\ast $-ideal is an integral $\ast $%
-ideal that is maximal among proper integral $\ast $-ideals. Let $\ast $-Max$%
(D)$ be the set of maximal $\ast $-ideals of $D$. For a star operation $\ast 
$ of finite character, it is well known that a maximal $\ast $-ideal is a
prime ideal; every proper integral $\ast $-ideal is contained in a maximal $%
\ast $-ideal; and $\ast $-Max$(D)\neq \emptyset $ if $D$ is not a field. For
a star operation $\ast $ two ideals $A,B$ may be called $\ast $-comaximal if 
$(A,B)^{\ast }=D.$ Indeed if $\ast $ is of finite character then two ideals $%
A,B$ are $\ast $-comaximal if, and only if, $A,B$ do not share (being in) a
maximal $\ast $-ideal $M.$ Thus integral ideals $A_{1},A_{2},...,A_{n}$ are $%
\ast $-comaximal to an ideal $B$ if and only if $(A_{1}A_{2}...A_{n},B)^{%
\ast }=D.$ Next, $I_{w}=\bigcap_{M\in t\text{-Max}(D)}ID_{M}$ for every $%
I\in F(D)$ and $I_{w}D_{M}=ID_{M}$ for every $I\in F(D)$ and $M\in t$-Max$%
(D) $. A $\ast $-operation that gets defined in terms of maximal $\ast $%
-ideals is denoted by $\ast _{w}$ and it is defined as follows: For $I\in
F(D),$ and $I^{\ast _{w}}=\dbigcap\limits_{M\in \ast _{s}\text{-Max}%
(D)}ID_{M}.$ This operation was introduced in \cite{AC} where it was
established that for any star operation $\ast ,$ $\ast _{w}$ is a star
operation of finite character and $\ast _{w}$-Max$(D)=\ast _{s}$-Max$(D)$
and $\ast _{w}\leq \ast ,$ according to, again, \cite{AC}. An integral
domain $D$ is a P$\ast $MD if and only if $D_{M}$ is a valuation domain for
every maximal $\ast $-ideal $M$ of $D$, \cite{HMM}. Next, as the $\ast $%
-product $(IJ)^{\ast }$ of two $\ast $-invertible $\ast $-ideals is again $%
\ast $-invertible it is easy to see that $Inv_{\ast }(D)=\{I:$ $I$ is a $%
\ast $-invertible $\ast $-ideal of $D\}$ is a group under $\ast $%
-multiplication with $P(D)$ the group of nonzero principal fractional ideals
of $D$ as its sub group. The quotient group $Inv_{\ast }(D)/P(D)$ is called
the $\ast $-class group of $D,$ denoted by $Cl_{\ast }(D).$ The $\ast $%
-class groups were introduced and studied by D.F. Anderson in \cite{An} as a
generalization of the $t$-class groups introduced in \cite{B}, \cite{Z} and
further studied in \cite{BZ}. It was shown in \cite{An}, in addition to many
other insightful results, that if $\ast _{1}\leq \ast _{2}$ are two star
operations then $Cl_{\ast _{1}}(D)\subseteq Cl_{\ast _{2}}(D).$

In section 2 we discuss and establish the main features of the general
theory as described in the introduction and in section 3 we discuss the
various examples or special cases of the $\ast $-SH domains, while in
section 4 we discuss weaker or restricted theories such as weakly factorial
domains and almost weakly factorial domains etc. where the $\ast $-homog
ideals have certain properties under special circumstances. In this section
we also give examples, as those examples do not frequent the general scene
as often as those discussed in section 3.

\section{Main Theory}

For a start, to save on space, let us agree that throughout hence, $\ast $
will denote a star operation of finite character, defined on $D,$ and that $%
D $ will be reserved for an integral domain. Let's also recall from the
introduction that a domain $D$ is a $\ast $-SH domain, if for every nonzero
non unit $x$ of $D$ the ideal $xD$ is expressible as a $\ast $-product of $%
\ast $-homog ideals of $D.$ We start with explaining what an ideal being $%
\ast $-homog means.

Let's, for a start, recall that we call a nonzero $\ast $-ideal of finite
type a $\ast $-homogeneous ($\ast $-homog) ideal, if $I\subsetneq D$ and $%
(J+K)^{\ast }\neq D$ for every pair $D\supsetneq J,K\supseteq I$ of proper $%
\ast $-ideals of finite type$.$

\begin{proposition}
\label{Proposition A} Let $I$ be a $\ast $-homog ideal of $D.$ Define $%
M(I)=\{x\in D:$ $(x,I)^{\ast }\neq D\}.$ Then $M(I)$ is the unique maximal $%
\ast $-ideal of $D$ containing $I.$
\end{proposition}

\begin{proof}
Clearly $M(I)$ is an ideal. For, let $x,y\in M(I).$ Then $(x,I)^{\ast }\neq
D\neq (y,I)^{\ast }.$ But then, by definition, $((x,I)^{\ast }+(y,I)^{\ast
})^{\ast }\neq D.$ Now $((x,I)^{\ast }+(y,I)^{\ast })^{\ast
}=((x,y)+I)^{\ast }.$ Also as $((x,y)+I)^{\ast }=((x,y)^{\ast }+I)^{\ast },$
we conclude that for each $a\in (x,y)^{\ast }$ we have $(a,I)^{\ast }\neq D$
and so $x,y\in M(I)$ implies $(x,y)^{\ast }\subseteq M(I).$ Next, let $%
x_{1},x_{2},...,x_{n}\in M(I).$ Then as we have already seen $%
(x_{1},x_{2})^{\ast }\subseteq M(I).$ Suppose that we have shown that for $%
x_{1},x_{2},...,x_{n-1}\in M(I)$ we have $(x_{1},x_{2},...,x_{n-1})^{\ast
}\subseteq M(I).$ Then by definition $%
(((x_{1},x_{2},...,x_{n-1}),I)+(x_{n},I))^{\ast }\neq D$ and so $D\neq
((x_{1},x_{2},...,x_{n}),I)^{\ast }=$ $((x_{1},x_{2},...,x_{n})^{\ast
},I)^{\ast }$ which of course means that for each $\alpha \in
(x_{1},x_{2},...,x_{n})^{\ast },$ $(\alpha ,I)^{\ast }\neq D,$ i.e. for each
positive integer $n,$ $x_{1},x_{2},...,x_{n}\in M(I)$ implies that $%
(x_{1},x_{2},...,x_{n})^{\ast }\subseteq M(I).$ That is, $M(I)$ is a $\ast $%
-ideal, because $\ast $ is of finite character. Indeed as for each $x\in I$ $%
(x,I)^{\ast }=I\neq D$ we have $I^{\ast }\subseteq M(I).$ Now let $P$ be a
maximal $\ast $-ideal containing $M(I),$ then since for each $x\in P$ we
must have $(x,I)^{\ast }\neq D,$ $P=M(I)$ and so $M(I)$ is the unique
maximal $\ast $-ideal containing $I.$
\end{proof}

In \cite{AAZ} a finitely generated nonzero ideal $I$ was called rigid if $I$
belonged to exactly one maximal $t$-ideal, and the maximal ideal containing
a rigid ideal was in turn called potent in \cite{ACZ}. Taking a cue from
that a finitely generated ideal $I$ was called $\ast $-rigid, in \cite{HZ2},
if $I$ belongs to exactly one maximal $\ast $-ideal. (In \cite{DZ} a $\ast $%
-rigid ideal was called homogeneous.) Let us call $M(I),$ defined in the
above proposition, the maximal $\ast $-ideal spawned by $I.$

\begin{corollary}
\label{Corollary B} Let $I$ be a $\ast$-homog ideal. If $I$ is contained in
a prime ideal $Q$ that is contained in some $\ast$-ideal then $Q\subseteq
M(I).$
\end{corollary}

For the record we state and prove the following easy to prove result.

\begin{corollary}
\label{Corollary C} A nonzero $\ast$-ideal $I$ of finite type is a $\ast $%
-homog ideal if and only if $I$ is a $\ast$-rigid ideal.
\end{corollary}

\begin{proof}
That a $\ast $-homog ideal $I$ is $\ast $-rigid follows from the fact that $%
M(I)$ is the unique maximal $\ast $-ideal containing $I.$ Conversely $I$ is
rigid because $I$ is contained in a unique maximal ideal $P$ and let $A$ be
a $\ast $-ideal containing $I$ such that $A^{\ast }\neq D$, then $A$ must be
in $P$ and in no other maximal $\ast $-ideal because $A$ contains $I.$ So if 
$J$ and $K$ are two proper $\ast $-ideals of finite type containing $I$ then
both $J,K$ are contained in $P$ and hence $(J+K)^{\ast }\neq D.$ So, by the
definition, every $\ast $-rigid ideal is $\ast $-homog.
\end{proof}

\begin{remark}
\label{Remark D} (1) The converse in Corollary \ref{Corollary C} gives the
reason why this is the case that if an ideal $A$ contains a power of a
principal prime $pD$ then $A_{t}\neq D$ only if $A\subseteq pD.$ This is
because each positive power of a principal prime indeed generates a rigid
ideal. It also indicates that if two $\ast $-homog ideals are such that $%
M(I)\neq M(J)$ then $(I+J)^{\ast }=D$ that is $I$ and $J$ do not share a
maximal $\ast $-ideal, which is obvious. (2). The idea of a $\ast $ homog
ideal comes from \cite{DZ}. (3) There may be a question as to why use $\ast $%
-homog when we already have $\ast $-rigid. My reason is partly choice and
partly the fact that when we say, \textquotedblleft $I$ is $\ast $%
-rigid\textquotedblright\ we have to declare the maximal $\ast $-ideal $M$
it belongs to. On the other hand, when we say,\textquotedblleft $I$ is $\ast 
$-homog\textquotedblright\ we do not have to worry about that, as $I$
determines its own maximal $\ast $-ideal $M(I).$
\end{remark}

\begin{proposition}
\label{Proposition E} Let $I,J$ be two $\ast $-homog ideals and let $K$ be a 
$\ast $-homog ideal such that $K\subseteq M(I)\cap M(J).$ Then $%
M(I)=M(J)=M(K).$ Consequently if $I$ and $J$ spawn two distinct maximal $%
\ast $-ideals then $M(I)\cap M(J)$ does not contain a $\ast $-homog ideal.
\end{proposition}

\begin{proof}
Since $K\subseteq M(I)$ and $M(I)$ is a $\ast $-ideal, for each $x\in M(I),$ 
$(x,K)^{\ast }\subseteq M(I)$ and so for each $x\in M(I)$ $(x,K)^{\ast }\neq
D.$ But then $M(I)\subseteq M(K).$ But as $M(I)$ is a maximal $\ast $-ideal
we have $M(I)=M(K).$ Similarly for $M(J).$ The consequently part is obvious.
\end{proof}

We shall call two or more $\ast$-homog ideals similar if they spawn (are
contained in) the same maximal $\ast$-ideal. Indeed it is easy to deduce
from the criterion of similarity that similarity of $\ast$-rigid ideals is
an equivalence relation. If $A$ and $B$ are $\ast$-homog ideals spawning
distinct maximal $\ast$-ideals we may call $A,B$ dissimilar. Also by Remark %
\ref{Remark D}, dissimilar means $\ast$-comaximal.

\begin{corollary}
\label{Corollary F} If $I$ and $J$ are two similar $\ast$-homog ideals then $%
(IJ)^{\ast}$ is a $\ast$-homog ideal similar to them. Consequently any $\ast 
$-product of mutually similar $\ast$- homog ideals is similar to each of
them.
\end{corollary}

\begin{proof}
Suppose that $(IJ)^{\ast }$ is contained in another maximal $\ast $-ideal $P$
then as, say, $I\subseteq P$ we have $P=M(I)$. Similarly if there is a
maximal $\ast $-ideal $Q$ containing, say, $J$ then $Q=M(J)=M(I).$ We can
deal with the consequently part using induction.
\end{proof}

\begin{proposition}
\label{Proposition G} Let $I$ be a $\ast $-homog ideal of $D.$ Then $%
ID_{M(I)}\cap D=I.$
\end{proposition}

\begin{proof}
Note that since $I$ is a $\ast $-ideal, $I=I^{\ast
_{w}}=\dbigcap\limits_{P\in \ast \text{-}\max (D)}I^{\ast }D_{P},$ because $%
\ast _{w}\leq \ast .$ But since $I$ is $\ast $-homog $ID_{P}=D_{P}$ for all
maximal $\ast $-ideals $P$ other than $M(I),$ $I=(I)^{\ast
_{w}}=ID_{M(I)}\cap (\dbigcap\limits_{P\in \ast \text{-}\max (D)\backslash
M(I)}D_{P}).$ So, $I=ID_{M(I)}\cap D\neq D$ as $I\subseteq M(I).$
\end{proof}

\begin{corollary}
\label{Corollary H} Let $I$ be a $\ast $-homog ideal of $D$ and let $A,B$ be
ideals of $D$ such that $A,B$ are $\ast $-comaximal, i.e. $(A+B)^{\ast }=D,$
(i) If $AB\subseteq I$ then $A\subseteq I$ or $B\subseteq I.$ (ii) If $B$ is 
$\ast $-homog and $(A+B)^{\ast }=D,$ then $(AB)^{\ast }D_{M(B)}\cap D=B.$
\end{corollary}

\begin{proof}
(i) Clearly if $AB\subseteq I$ then $AB\subseteq M(I)$ which implies $%
A\subseteq M(I)$ or $B\subseteq M(I)$ but not both, because $A,B$ are $\ast $%
-comaximal. Now if $A\subseteq M(I)$ then $AB\subseteq I$ implies $%
ABD_{M(I)}=AD_{M(I)}\subseteq ID_{M(I)}.$ This in turn implies that $%
A\subseteq AD_{M(I)}\cap D\subseteq ID_{M(I)}\cap D=I.$ (ii) As $AB\subseteq
(AB)^{\ast }$ and as $A\nsubseteq M(B)$ we conclude that $%
BD_{M(B)}=ABD_{M(B)}\subseteq (AB)^{\ast }D_{M(B)}$ and so $B=BD_{M(B)}\cap
D\subseteq (AB)^{\ast }D_{M(B)}\cap D.$ But already $AB\subseteq B$ and so $%
(AB)^{\ast }\subseteq B$ we have $(AB)^{\ast }D_{M(B)}\cap D\subseteq
BD_{M(B)}\cap D=B.$
\end{proof}

\begin{proposition}
\label{Proposition K} If an ideal $A$ is expressible as a $\ast$-product of $%
\ast$-homog ideals, then $A$ is expressible, uniquely, up to order, as a $%
\ast$-product of mutually $\ast$-comaximal $\ast$-homog ideals.
\end{proposition}

\begin{proof}
Let $A=(I_{1}I_{2}...I_{m})^{\ast }$. Pick a $\ast $-homog factor say $I_{1}$
and collect all the $\ast $-homog factors of $A$ that are similar to $I_{1}$%
. Next suppose that by a relabeling $I_{1},I_{2},...,I_{n_{1}}$ are all
similar to $I_{1}$and all the remaining ideals are dissimilar to $I_{1}$ and
hence to all of $I_{1},I_{2},...,I_{n_{1}}.$Then by Corollary \ref{Corollary
F}, $J_{1}=(I_{1}I_{2}...I_{n_{1}})^{\ast }$ is a $\ast $-homog ideal. So $%
A=(J_{1}I_{n_{1}+1}...I_{m})^{\ast }$ where none of the $I_{n_{1}+i}$ spawns
the same maximal $\ast $-ideal as $M(J_{1}).$ Now collect all the factors
similar to $I_{n_{1}+1}$ and suppose, by a relabeling, that those factors
are all of $I_{n_{1}+1},I_{n_{1}+2},...,I_{n_{2}}$ and the rest are all
dissimilar to $I_{n_{1}+1}.$ Then by setting $J_{2}=$ $%
(I_{n_{1}+1}I_{n_{1}+2}...I_{n_{2}})^{\ast }$ we have $%
A=(J_{1}J_{2}I_{n_{2}+1}...I_{m})^{\ast }$ where $J_{1},J_{2}$ are $\ast $%
-comaximal to each other and all remaining $I_{i}.$ Continuing thus we can
end up with $A=(J_{1}J_{2}...J_{r})^{\ast }$ where $J_{i}$ are mutually $%
\ast $-comaximal. Since $J_{i}$ are mutually $\ast $-comaximal, and so no
two can be in the same maximal $\ast $-ideal, $AD_{M(J_{i})}\cap D=J_{i},$
by Corollary \ref{Corollary H}. Now suppose that $A$ has another expression $%
A=(K_{1}K_{2}...K_{s})^{\ast }$ as a $\ast $-product of mutually $\ast $%
-comaximal $\ast $-homog ideals. Then $A=(J_{1}J_{2}...J_{r})^{\ast
}=(K_{1}K_{2}...K_{s})^{\ast }.$ As $A\subseteq J_{1}\subseteq M(J_{1}),$ we
must have $K_{i}\subseteq M(J_{1})$ for some $i$, and as $K_{j}$ are mutually%
$\ast $-comaximal, $K_{j}\nsubseteq M(J_{1})$ for $j\neq i.$ But then $%
M(K_{i})=M(J_{1}).$ Next by Corollary \ref{Corollary H}, $AD_{M(J_{1})}\cap
D=J_{1}$ and $AD_{M(K_{i})}\cap D=K_{i}.$ But as $M(K_{i})=M(J_{1}),$ we
conclude that $J_{1}=K_{i}.$ Thus for each $J_{i}$ there is a $K_{j}$ such
that $J_{i}=K_{j}$ and $r\leq s.$ Indeed as $K_{j}$ are mutually $\ast $%
-comaximal, there is a unique $K_{j}$ to each $J_{i}.$ Similarly, starting
with $K$s from the right side we can show that $s\leq r,$ thus establishing
that $r=s.$
\end{proof}

\begin{corollary}
\label{Corollary L} Let $D$ be a $\ast $-SH domain, then each principal
ideal $xD$ generated by a nonzero non unit is uniquely expressible as a $%
\ast $-product of mutually $\ast $-comaximal $\ast $-homog ideals, each of
which is $\ast $-invertible. Also, if $M$ is a maximal $\ast $-ideal
containing $x$ then $xD_{M}\cap D$ is a $\ast $-invertible $\ast $-homog
ideal.
\end{corollary}

\begin{proof}
Since $D$ is a $\ast $-SHD, for every nonzero non unit $x\in D,$ $%
xD=(J_{1}J_{2}...J_{s})^{\ast }$ where each $J_{i}$ is a $\ast $-homog
ideal. By Proposition \ref{Proposition K} we can write $%
xD=(I_{1}I_{2}...I_{r})^{\ast }$ where $I_{i}$ are mutually $\ast $%
-co-maximal $\ast $-homog ideals. Also as $xD$ is $\ast $-invertible, each
of $I_{i}$ is $\ast $-invertible. Finally if $M$ is a maximal $\ast $
containing $x$ then because $I_{i}$ are mutually $\ast $-comaximal $M$
contains exactly one of the $I_{i},$ say $I_{k}.$ So $M=M(I_{k}).$ But then
by Corollary \ref{Corollary H} applies and we get $xD_{M}\cap D=I_{k}.$
\end{proof}

Recall that for a finite character star operation $\ast$ defined on $D,$ $D$
is of finite $\ast$-character if every nonzero non unit element of $D$
belongs to at most a finite number of maximal $\ast$-ideals. Recall also
that if $P$ and $Q$ are two prime ideals of $D$ such that no nonzero prime
ideal is contained in $P\cap Q$ then $D_{P}D_{Q}=K$ the quotient field of $D$
\cite[Lemma 4.1]{AZ2}.

\begin{theorem}
\label{Theorem M}The following are equivalent for an integral domain $D.$
(1) $D$ is a $\ast$-SH domain, (2) $D$ is of finite $\ast$-character and for
every pair $P,Q$ of distinct maximal $\ast$-ideals $P\cap Q$ does not
contain a nonzero prime ideal.
\end{theorem}

\begin{proof}
(1) $\Rightarrow $ (2). Suppose that $D$ is a $\ast $-SH domain and let $x$
be a nonzero non unit of $D.$ Then $xD=(I_{1}I_{2}...I_{r})^{\ast }$ where $%
I_{i}$ are $\ast $-homog ideals. Since each $I_{i}$ $\subseteq M(I_{i})$
which is a unique maximal $\ast $-ideal, by Proposition \ref{Proposition A},
we conclude that $x$ belongs to at most $r$ maximal $\ast $-ideals. Also if $%
P$ is a maximal $\ast $-ideal of $D$ then for $x\in P\backslash \{0\}$ $%
xD=(I_{1}I_{2}...I_{r})^{\ast }$ and so at least one of $I_{i},$ say $%
I_{j}\subseteq P.$ But then by Corollary \ref{Corollary B}, $P\subseteq
M(I_{j}).$ Since $P$ is a maximal $\ast $-ideal $P=M(I_{j}).$ Thus for each
maximal $\ast $-ideal $P$ of a $\ast $-SHD $D,$ there is a $\ast $-homog
ideal $I$ such that $P=M(I).$ Now let $P$ and $Q$ be two distinct maximal $%
\ast $-ideals in a $\ast $-SH domain $D$. As we have established above,
there exist $\ast $-homog ideals $I,J$ such that $P=M(I)$ and $Q=M(J).$ Now
suppose that there is a nonzero prime ideal $m\subseteq P\cap Q.$ Then as $m$
is a nonzero prime ideal, $m$ contains a nonzero element and hence a $\ast $%
-homog ideal $A,$ which is impossible by Proposition \ref{Proposition E},
because $P=M(I)$ and $Q=M(J)$ are distinct. We next show (2) $\Rightarrow $
(1). Suppose that $D$ is of finite $\ast $ character and that no two maximal 
$\ast $-ideals $P,Q$ contain a nonzero prime ideal. Let $x$ be a nonzero non
unit element of $D.$ Let $T=\{P_{1},P_{2},...,P_{r}\}$ be the set of all the
maximal $\ast $-ideals containing $x.$ Then $(x)=\dbigcap\limits_{P\in \ast 
\text{-}Max(D)}xD_{P}=xD_{P_{1}}\cap xD_{P_{2}}\cap ...\cap xD_{P_{r}}\cap
(\dbigcap\limits_{Q\in \ast \text{-}Max(D)\backslash T}D_{Q})$ $%
=(xD_{P_{1}}\cap xD_{P_{2}}\cap ...\cap xD_{P_{r}})\cap
D=\dbigcap\limits_{i=1}^{r}(xD_{P_{i}}\cap D)$ . We now proceed to show that 
$xD_{P_{i}}\cap D$ is contained in $P_{i}$ and to no other maximal $\ast $
-ideal for each $i=1,..,r.$ Indeed for any maximal $\ast $-ideal $Q$ other
than $P_{i}$ we have $D_{P_{i}}D_{Q}=K$ \cite[Lemma 4.1]{AZ2} we have $%
(xD_{P_{i}}\cap D)D_{Q}=xD_{P_{i}}D_{Q}\cap D_{Q}=K\cap D_{Q}=D_{Q}.$ So $%
xD_{P_{i}}\cap D$ is not contained in any maximal $\ast $-ideal other than $%
P_{i}.$ Using this piece of information we see that $(\dprod%
\limits_{i=1}^{r}(xD_{P_{i}}\cap D))^{\ast
_{w}}=(\dbigcap\limits_{i=1}^{r}xD_{P_{i}}\cap D)^{\ast
_{w}}=\dbigcap\limits_{i=1}^{r}xD_{P_{i}}\cap D=(x)$ and as $\ast _{w}\leq
\ast $ we have $(x)=$ $(\dprod\limits_{i=1}^{r}(xD_{P_{i}}\cap D))^{\ast }.$
That is each $(xD_{P_{i}}\cap D)$ is $\ast $-invertible and hence of finite
type and consequently is a $\ast $-homog ideal, being also a $\ast $-ideal.
Thus for each nonzero non unit $x$ of $D,$ $xD$ is expressible as a finite $%
\ast $-product of $\ast $-homog ideals and $D$ is a $\ast $-SH domain.
\end{proof}

The proof of (2) $\Rightarrow $ (1) of Theorem \ref{Theorem M} shows that
Theorem \ref{Theorem M} could have been replaced by another interesting
result, if we were to use the terminology of \cite{AZ2}. The terminology can
be described as follows. Let $\mathcal{F}$ $=\{P_{\alpha }:\alpha \in I\}$
be a family of nonzero prime ideals of $D.$ $\mathcal{F}$ is called a
defining family of $D$ if $D=\dbigcap\limits_{\alpha \in I}D_{P_{\alpha }}$
. The defining family $\mathcal{F}$ is of finite character if no nonzero non
unit of $D$ belongs to infinitely many members of $\mathcal{F}$. We may call
the defining family $\mathcal{F}$ independent if no two members of $\mathcal{%
F}$ contain a common nonzero prime ideal. The function $\ast _{\mathcal{F}}$
on $F(D)$ defined by $A\mapsto A^{\ast _{\mathcal{F}}}=\dbigcap\limits_{%
\alpha \in I}AD_{P_{\alpha }}$ is called a star operation induced by the
family $\{D_{P_{\alpha }}\}$ of localizations at members of $\mathcal{F}$.
(We shall, in what follows, introduce concepts to facilitate reading of the
paper.) In \cite{AZ2}, an integral ideal $A$ of $D$ was called
unidirectional if $A$ belongs to a unique member of the defining family $%
\mathcal{F}$ of primes. With this terminology it was shown in \cite[Theorem
2.1]{AZ2} that if $D$, $\mathcal{F}$, $\ast _{\mathcal{F}}$ are defined as
above and if $\ast _{\mathcal{F}}$ is of finite character then the family $%
\mathcal{F}$ is independent of finite character if and only if every nonzero
non unit element $x$ of $D,$ $xD$ is expressible as a $\ast _{\mathcal{F}}$%
-product of a finite number of unidirectional ideals.

Now, if we match $\mathcal{F}$ with $\ast$-$Max(D),$ $\ast_{\mathcal{F}}$
gets matched with $\ast_{w}$ and unidirectional ideal with $\ast$-homog
ideal we can restate Theorem 3.3 of \cite{AZ2} as the following result.

\begin{theorem}
\label{Theorem N} Let $\ast $ be a star operation of finite character
defined on an integral domain $D.$ Then the following are equivalent: 1. $D$
is of finite $\ast $-character and for any two distinct $P,Q\in \ast $-$%
Max(D)$ $P\cap Q$ does not contain a nonzero prime ideal, 2. every nonzero
prime ideal of $D$ contains an element $x$ such that $xD$ is a $\ast _{w}$%
-product of $\ast $-homog ideals (Note that as $\ast _{w}\leq \ast $,
\textquotedblleft $\ast _{w}$-product\textquotedblright\ can be replaced by
\textquotedblleft $\ast $-product\textquotedblright , here.), 3. every
nonzero prime ideal of $D$ contains a $\ast $-homog $\ast $-invertible $\ast 
$-ideal, 4. for $P\in \ast $-$Max(D)$ and $0\neq $ $x\in P$, $xD_{P}\cap D$
is a $\ast $-invertible and $\ast $-homog ideal (In \cite{AZ2} $0\neq $ $%
x\in D$ was mistakenly typed in place of $0\neq $ $x\in P.),$ 5. no pair of
distinct maximal $\ast $-ideals contains a nonzero prime ideal and for any
nonzero ideal $A$ of $D$, $A^{\ast _{w}}$ is of finite type. whenever $%
AD_{P} $ is finitely generated for all $P$ in $\ast $-$Max(D)$.
\end{theorem}

Note that Theorem \ref{Theorem M} proves the equivalence of (1) and (2) of
Theorem \ref{Theorem N} and that is grounds enough to include Theorem \ref%
{Theorem N} as part of this paper. On the other hand the theory developed in 
\cite{AZ2} is not enough to take care of the more general approach in this
paper. There is, of course, another important difference. While \cite{AZ2}
takes care of independent rings of Krull type and Krull domains by requiring
that for each $P\in \mathcal{F}$, $D_{P}$ is a valuation domain, and
requiring for Krull domains that $D_{P}$ is a rank one DVR for each $P\in 
\mathcal{F}$, the theory presented here lets us define a $\ast $-homog ideal
to fit the picture. For instance we can define, as we show in the following,
a $\ast $-homog ideal to establish the theory of independent rings of Krull
type, or of Krull domains etc. In each case, obviously, Theorem \ref{Theorem
M} and Theorem \ref{Theorem N} are ready proved and all we need show is that
the resulting theory has the distinctive feature that we claim it has. In
what follows, in the next section we make a few demonstrations that list the
variations of the definition of $\ast $-homog ideals and the domains they
lead to.

In view of a comment, at the start of section 3, in \cite{AZ2} we may also
call the domains characterized in Theorems \ref{Theorem M}, \ref{Theorem N} $%
\ast $-h-local, as the domains whose principal ideals generated by nonzero
non units are expressible as $\ast $-products of finitely many $\ast $-homog
ideals, noting that when $\ast =d$ we have the usual definition of h-local
domains of Matlis \cite{Mat} and when $\ast =t$ we have what we termed as
weakly Matlis domains in \cite{AZ2}. The interesting part of this approach
is that, as we demonstrate below, we can redefine the $\ast $-homog ideals
to fit the various special cases of $\ast $-SH domains.

\section{Clones or examples of $\ast $-SHDs}

Let's call $D$ a $\ast $-weakly Krull domain ($\ast $-WKD) if $D$ is a $\ast 
$-SHD such that each maximal $\ast $-ideal $P$ of $D$ is of height $1.$
These domains are known as weakly Krull domains and were first studied in 
\cite{AMZ}.

\begin{definition}
\label{Definition P1} Call a $\ast $-homog ideal $I,$ $\ast $-homog of type $%
1$, if for every $x\in M(I)\backslash \{0\}$ there is a positive integer $n$
such that $x^{n}D_{M(I)}\cap D\subseteq I.$ Also call a domain $D$ a $\ast $%
-SH domain of type $1$ if for every nonzero non unit $x$ of $D$, $xD$ is a $%
\ast $-product of finitely many $\ast $-homog ideals of type $1.$
\end{definition}

Indeed if $I$ and $J$ are $\ast $-homog of type $1$ then so is $(IJ)^{\ast
}. $ This is because $(IJ)^{\ast }$ is a $\ast $-homog ideal similar to both 
$I$ and $J,$ to start with. So, $M(I)=M(J).$ Now let $x\in M(I)\backslash
\{0\}.$ Then for some positive integers $m,n$ we have $x^{m}D_{M(I)}\cap
D\subseteq I $ and $x^{n}D_{M(I)}\cap D\subseteq J.$ This gives $%
(x^{m}D_{M(I)}\cap D)(x^{n}D_{M(I)}\cap D)\subseteq IJ\subseteq (IJ)^{\ast
}. $ But then $(IJ)^{\ast }=(IJ)^{\ast }D_{M(I)}\cap D\supseteq
(x^{m}D_{M(I)}\cap D)(x^{n}D_{M(I)}\cap D))D_{M(I)}\cap D$ $%
=x^{m+n}D_{M(I)}\cap D,$ for each $x\in M(I).$

\begin{theorem}
\label{Theorem P2} Let $D$ be an integral domain and suppose that $D$ is a $%
\ast $-SHD of type $1$. Then $D$ is a $\ast $-WKD. Conversely if $D$ is a $%
\ast $-WKD, then every nonzero proper principal ideal of $D$ is expressible
as a $\ast $-product of finitely many $\ast $-homog ideals of type $1,$ i.e. 
$D$ is a $\ast $-SHD of type $1.$
\end{theorem}

\begin{proof}
All we need prove is that if every nonzero proper principal ideal of $D$ is
expressible as a finite $\ast $-product of $\ast $-homog ideals of type $1$
then for each maximal $\ast $-ideal $M$ we have $ht(M)=1.$ For this let us
first observe that if $D$ is a $\ast $-SHD with $M$ a maximal $\ast $-ideal
of $D$ then by Corollary \ref{Corollary L}, $xD_{M}\cap D$ is a $\ast $%
-invertible $\ast $-homog ideal. Indeed as in the proof of Corollary \ref%
{Corollary L}, $xD=(I_{1}I_{2}...I_{r})^{\ast }$ where $I_{i}$ are mutually $%
\ast $-comaximal and of type $1,$ because $x$ is a product of $\ast $-homog
ideals of type $1.$ Also as $M$ is a maximal $\ast $-ideal, $xD_{M}\cap
D=I_{i}$ for some $i,$ by the proof of Corollary \ref{Corollary L}. Now let $%
aD_{M},bD_{M}$ be two nonzero non units in $D_{M}.$ We can assume that $%
a,b\in D.$ So that $aD_{M}\cap D$ is of type $1$ and $b\in M$ and so by
definition $b^{m}D_{M}\cap D\subseteq aD_{M}\cap D$, for some positive
integer $m.$ But then $(b^{m}D_{M}\cap D)D_{M}\subseteq (aD_{M}\cap D)D_{M}$
and this means $b^{m}D_{M}\subseteq aD_{M}$ . Since $a$ and $b$ are
arbitrary, we conclude that $MD_{M}$ is of height one. (We have $%
b^{m}D_{M}\subseteq aD_{M},$ for some $m,$ so $bD_{M}$ belongs to every
nonzero prime ideal $aD_{M}$ belongs to. Similarly the other way around and
indeed this forces $M$ to be of height one.) This makes the, otherwise, $%
\ast $-h-local domain a $\ast $-WKD. Conversely if $D$ is a $\ast $-WKD.
Then as we know from Theorems \ref{Theorem M} and \ref{Theorem N} every
ideal generated by a nonzero non unit of $D$ is a $\ast $-product of $\ast $%
-homog ideals and we need to show these $\ast $-homog ideals are of type $1.$
Now all we need do is show that for each nonzero $x$ and a maximal $\ast $%
-ideal $M,$ containing $x,$ the ideal $xD_{M}\cap D$ is of type $1$. For
this let $a\in M\backslash (0).$ Then $aD_{M}$ and $xD_{M}$ are nonzero non
units of $MD_{M}$ which is of height one and so $a^{m}D_{M}\subseteq xD_{M}$
for some positive integer $m.$ But then $a^{m}D_{M}\cap D\subseteq
xD_{M}\cap D$ which makes $xD_{M}\cap D$ of type $1.$
\end{proof}

Recall that, as we hinted in relation with Theorem \ref{Theorem N} that $D$
is a Krull domain if $D=\dbigcap D_{P}$ where the intersection is locally
finite and each $D_{P}$ is a discrete valuation domain. Let's call a $\ast $%
-WKD a $\ast $-Krull domain if for each maximal $\ast $-ideal $P$ of $D$,
the localization $D_{P}$ is a discrete rank one valuation domain. Now note
that a $\ast $-WKD $D$ is $\ast $-Krull if and only if every maximal $\ast $%
-ideal of $D$ is $\ast $-invertible. Since if $\ast $ is of finite type then
every $\ast $-invertible $\ast $-ideal is a $t$-invertible $t$-ideal \cite[%
Theorem 1.1]{z00} and for a $t$-invertible prime $t$-ideal $P$ of height $1,$
$D_{P}$ is a discrete valuation domain, because $t$-invertible extends to $t$%
-invertible in localizations \cite[page 436, consequence (a)]{z00} $PD_{P}$
is $t$-invertible and because $PD_{P}$ is of height one, $PD_{P}$ is a $t$%
-ideal and in a $t$-local domain (i.e. maximal ideal is a $t$-ideal.) $t$%
-invertible is principal \cite[Proposition 1.12]{ACZ}. So, in view of the
definitions of a Krull domain, a $\ast $-Krull domain is a Krull domain. It
appears that a definition that links the $\ast $-homog ideal with this fact
can be worded as below.

\begin{definition}
\label{Definition P3} Call a $\ast $-homog ideal $I$, of type $2$ if for
some positive integer $n,$ $I=((M(I))^{n})^{\ast }.$ Also call a domain $D$
a $\ast $-SH domain of type $2$ if for every nonzero $x$ in $D$, $xD$ is
expressible as a $\ast $-product of a finite number of $\ast $-homog ideals
of type $2.$
\end{definition}

Indeed if $I$ and $J$ are both $\ast $-homog of type $2$ then $(IJ)^{\ast
}=((M(I))^{n}(M(I))^{p})^{\ast }$ and that makes $(IJ)^{\ast }$ of type $2.$

\begin{theorem}
\label{Theorem P4} Let $D$ be an integral domain and suppose that $D$ is a $%
\ast $-SH domain of type $2$. Then $D$ is a $\ast $-Krull domain. Conversely
if $D$ is a $\ast $-Krull domain, then every nonzero proper principal ideal
of $D$ is expressible as a $\ast $-product of finitely many $\ast $-homog
ideals of type $2$.

\begin{proof}
Indeed a $\ast $-homog ideal $I$ that is of type $2$ is of type $1$ as well,
because if $x\in M(I)$, then $xD_{M(I)}\cap D\subseteq M(I)$ and so $%
x^{n}D\subseteq M(I)^{n}\subseteq I.$ But then $x^{n}D_{M}\cap D$ $\subseteq
ID_{M(I)}\cap I=I,$ by Proposition \ref{Proposition G}. So, $D$ is a $\ast $%
-WKD. Next, let $M$ be a maximal $\ast $-ideal and let $x$ be a nonzero
element in $M.$ Because $xD=(I_{1}I_{2}...I_{n})^{\ast }$ where each of the $%
I_{i}$ is a $\ast $-homog ideal of type $2$, and each of $I_{i}$ is $\ast $%
-invertible. Also at least one of $I_{i},$ say $I_{j},$ is contained in $M.$
But as $I_{j}$ is $\ast $-homog and as $M$ is a maximal $\ast $-ideal, $%
M=M(I_{j}).$ Finally as $I_{j}$ is of type $2,$ $I_{j}^{\ast }=(M^{n})^{\ast
}.$ This makes $M$ $\ast $-invertible. But as $\ast $ is of finite type, $M$
being $\ast $-invertible means $M$ is $t$-invertible and so is $MD_{M},$ 
\cite[page 436, consequence (a)]{z00}. Next as $MD_{M}$ is of height one, $%
MD_{M}$ is a $t$-ideal and $D_{M}$ is a $t$-local ring. But in a $t$-local
ring $t$-invertible is principal, \cite[Proposition 1.12]{ACZ}. But this
makes $D_{M}$ a one dimensional quasi local domain with maximal ideal
principal and so a rank one DVR. Now since $M$ was arbitrary and $D$ is of
finite $t$-character $D$ is $\ast $-Krull, as defined above. Conversely,
note that a $\ast $-Krull domain $D$ is a Krull domain, as we have already
established and every maximal $\ast $-ideal of $D$ is $\ast $-invertible and
hence a $t$-invertible $t$-ideal. Now it is well known that $D$ is a Krull
domain if and only if every proper principal ideal of $D$ is a $t$-product
of prime $t$-ideals \cite[Corollary 3.2]{AMZ}. So, $xD=(P_{1}...P_{n})_{t}.$
Moreover, as a $\ast $-Krull domain is a P$\ast $MD, because $D_{M}$ is a
valuation domain for every maximal $\ast $-ideal $M$, $\ast =t$ and thus $%
(P_{1}...P_{n})^{\ast }=(P_{1}...P_{n})_{t}=xD.$ Now as in a Krull domain
each prime $t$-ideal is a maximal $t$-ideal which is a maximal $\ast $-ideal
we can say that in a $\ast $-Krull domain every proper $\ast $-ideal is a $%
\ast $-product of maximal $\ast $-ideals. Finally as each maximal $\ast $%
-ideal $P$ in a $\ast $-Krull domain is a $\ast $-ideal of finite type,
being $\ast $-invertible, it is obviously $\ast $-homog of type $2.$
\end{proof}
\end{theorem}

\begin{definition}
\label{Definition Q} A nonzero integral $\ast $-ideal $I$ of finite type is
called $\ast $-super homogeneous ($\ast $-super homog) if (1) if each $\ast $%
-ideal of finite type containing $I$ is $\ast $-invertible and (2) For every
pair of proper integral $\ast $-ideals $A,B$ of finite type containing $I,$ $%
(A+B)^{\ast }\neq D.$
\end{definition}

\begin{remark}
\label{Z} Note that since every $\ast $-ideal of finite type is $\ast $%
-invertible in a P$\ast $MD a $\ast $-homog ideal is $\ast $-super homog in
a P$\ast $MD. Indeed, as the definition indicates, a $\ast $-super homog
ideal $I$ is a $\ast $-homog ideal such that each $\ast $-ideal of finite
type containing $I$ is $\ast $-invertible, in particular a $\ast $-super
homog ideal is $\ast $-invertible. Note that a $\ast $-super homog ideal is
a $\ast $-super rigid ideal of \cite{HZ2}. Some properties of $\ast $-super
rigid (i.e. $\ast $-super homog) ideals are given in \cite[Theorem 1.10]{HZ2}%
. We list them here in the language of the present paper, even though some
of them have been proved more in more general setting above.
\end{remark}

\begin{theorem}
\label{Theorem Q1} Let $I$ be a $\ast$-super homog ideal of $D$ and suppose
that $M=M(I).$
\end{theorem}

(1) If $A$ is a proper finitely generated ideal for which $A\supseteq I$,
then $A$ is $\ast$-super homog.

(2) If $J$ is a $\ast$-super homog ideal contained in $M$, then $I\subseteq
J $ $^{\ast}$or $J\subseteq I^{\ast}$.

(3) If $J$ is a $\ast$-super homog ideal contained in $M$, then $IJ$ is also
a $\ast$-super homog

ideal.

(4) $I^{n}$ is $\ast$-super homog for each positive integer $n$.

(5) If $D$ is local with maximal ideal $M$, then $I$ is comparable to each
ideal of

$D$, and $\cap_{n=1}^{\infty}$ $I^{n}$ is prime.

(6) $I^{\ast}$ $=ID_{M}\cap D$.

(7) $\cap_{n=1}^{\infty}($ $I^{n})^{\ast}$ is prime.

(8) If $P$ is a prime ideal of $D$ with $P\subseteq M$ and $I\nsubseteq P$,
then $P\subseteq \cap _{n=1}^{\infty }($ $I^{n})^{\ast }.$

\begin{proof}
of (2). Let $J$ be a $\star $-super rigid ideal contained in $M$, and set $%
C:=I+J$. Then $C$ is $\star $-invertible, and we have $(IC^{-1}+JC^{-1})^{%
\star }=R$. Note that $IC^{-1}\supseteq I$ and $JC^{-1}\supseteq J$, and
hence $IC^{-1}\nsubseteq M$ or $JC^{-1}\nsubseteq M$. Since $IC^{-1},JC^{-1}$
can be contained in no maximal $\star $-ideal of $R$ other than $M$, we must
have $(IC^{-1})^{\star }=R$ or $(JC^{-1})^{\star }=R$, that is, $C^{\star
}=I^{\star }$ or $C^{\star }=J^{\star }$. The conclusion follows easily.

Proof of (3). Let $K$ be a $\ast $-ideal of finite type such that $%
K\supseteq IJ.$ Then, as $J$ is $\ast $-invertible, being $\ast $-super
homog $(KJ^{-1})^{\ast }\supseteq I.$ Again as $I$ is $\ast $-super homog
and $(KJ^{-1})^{\ast }$ is a $\ast $-ideal of finite type we conclude that $%
(KJ^{-1})^{\ast }$ and hence $K$ is $\ast $-invertible. So, $(IJ)^{\ast }$
is such that each $\ast $-ideal of finite type containing $I$ is $\ast $%
-invertible. Now as a $\ast $-super homog ideal is a $\ast $-homog ideal and
the $\ast $-product of two similar $\ast $-homog ideals is a $\ast $-homog
ideal similar to them we have established that $(IJ)^{\ast }$ is $\ast $%
-super homog.
\end{proof}

Using the proof of (3) we can show that if $I_{1},I_{2},...,I_{r}$ are $\ast 
$-super homog ideals similar to each other then $(I_{1}I_{2}...I_{r})^{\ast
} $ is a $\ast $-super homog ideal similar to each of $I_{i}.$ Now Let $%
A=(J_{1}J_{2}...J_{n})^{\ast }$ be a $\ast $-product of a finite number of $%
\ast $-super homog ideals. Then as we can regroup them into classes of
similar $\ast $-super homog ideals as in the proof of Proposition \ref%
{Proposition K}, we can write $A=(K_{1}K_{2}...K_{m})^{\ast }$ where $%
K_{i}^{\ast }$ are mutually $\ast $-comaximal. But this expression is unique
being a $\ast $-product of mutually $\ast $-comaximal $\ast $-homog ideals
as shown in the proof of Proposition \ref{Proposition K}. We have thus
proved the following proposition.

\begin{proposition}
\label{Proposition Q1a} Let $J_{1},J_{2},...,J_{n}$ be a set of $\ast $%
-super homog ideals of a domain $D.$ Then the $\ast $-product $%
(J_{1}J_{2}...J_{n})^{\ast }$ can be expressed uniquely, up to order, as a $%
\ast $-product of mutually $\ast $-comaximal $\ast $-super homog ideals.
\end{proposition}

To make an efficient use of the material we have put together let us recall
that an integral domain $D$ an independent ring of Krull type (IRKT) if (1)
There is a family $F=$ $\{P_{\alpha }\}$ of prime ideals such that $%
D_{P_{\alpha }}$ is a valuation domain for each $P_{\alpha }\in F.$ (2) $%
D=\cap D_{P_{\alpha }}$ and the intersection is locally finite and (3) No
two members of $F$ contain as a subset a nonzero prime ideal of $D$.
Independent rings of Krull type were studied by Griffin \cite{Gr1}. Let us
call $D$ a $\ast $-independent ring of Krull type ($\ast $-IRKT), for a star
operation $\ast $ of finite type, if (i) $D_{P}$ is a valuation domain for
each maximal $\ast $-ideal $P,$ (ii) $D=\cap D_{P},$ the intersection is
locally finite and $P$ ranges over maximal $\ast $-ideals of $D$, (3) No two
distinct maximal $\ast $-ideals $P$ and $Q$ contain a nonzero prime ideal in
common. In other words a $\ast $-IRKT is a $\ast $-SH domain such that $%
D_{P} $ is a valuation domain for each maximal $\ast $-ideal $P$ of $D$. Now
recall, again, that a domain $D$ is called a P$\ast $MD if every finitely
generated nonzero ideal of $D$ is $\ast $-invertible and one of the
characterizations of a P$\ast $MD is that $D_{P}$ is a valuation domain for
each maximal $\ast $-ideal $P$ of $D$ \cite{HMM} and indeed a $\ast $-IRKT
is a P$\ast $MD, as we have noted above. Let's also note that a $\ast $-IRKT
is an IRKT and there is mixed opinion on whether there are any $\ast $%
-IRKTs, for finite type $\ast $ different from $d$ and $t.$ If $\ast =d$ the 
$\ast $-IRKT is indeed a Prufer domain. The situation gets complicated in
view of the fact that for any finite type star operation $\ast $ a $\ast $%
-invertible ideal is $t$-invertible \cite[Theorem 1.1]{z00}. In any case,
even $d$ and $t$ causing two different kinds of domains makes the case for
the use of a general $\ast $-operation approach sufficiently strong. We
shall call a $\ast $-SH domain whose $\ast $-homog ideals are $\ast $-super
homog a $\ast $-super SH domain. In general a $\ast $-homog ideal $I$ is a $%
\ast $-super homog ideal in a \.{P}$\ast $MD, with $\ast $ of finite type,
because every $\ast $-ideal $F$ of finite type containing $I$ is $\ast $%
-invertible.

We now proceed to show that if for every nonzero non unit $x$ of a domain $%
D, $ $xD$ is a $\ast $-product of $\ast $-super homog ideals then $D$ is a $%
\ast $-IRKT. Note that since a $\ast $-super homog ideal is $\ast $- homog,
a domain $D$ whose principal ideals generated by nonzero non units are $\ast 
$-products of $\ast $-super homog ideals is $\ast $-h-local to start with.
All we have to do is show that for each maximal $\ast $-ideal $P$ of $D$ the
localization $D_{P}$ is a valuation domain. For this all we need show is
that $xD_{P}$ and $yD_{P}$ are comparable for every pair of nonzero non
units $x,y$ in $D_{P}.$ As we can assume that $x,y\in D,$ we have that $%
xD_{P}\cap D,yD_{P}\cap D$ are $\ast $-homog by Corollary \ref{Corollary L}.
Indeed $xD=(I_{I}I_{2}...I_{r})^{\ast }$ where $I_{i}$ are mutually $\ast $%
-comaximal $\ast $-super homog ideals in the current situation.\ Then the
maximal $\ast $-ideal $P$ contains exactly one of the $\ast $-super homog
ideals $I_{i},$ say $I_{1}\subseteq P.$ That is $P=M(I_{1}).$ But then by
Corollary \ref{Corollary H} $I_{1}^{\ast }=xD_{M(I_{1})}\cap D=xD_{P}\cap D$
is a $\ast $-super homog ideal. Similarly $yD_{P}\cap D$ is a $\ast $-super
homog ideal. By (2) of Theorem \ref{Theorem Q1}, $xD_{P}\cap D\subseteq
yD_{P}\cap D$ or $xD_{P}\cap D\supseteq yD_{P}\cap D,$ because $xD_{P}\cap
D,yD_{P}\cap D$ are $\ast $-super homog ideals contained in the same maximal 
$\ast $-ideal. Now \textquotedblleft $xD_{P}\cap D,yD_{P}\cap D$
comparable\textquotedblright\ translates to $xD_{P}=(xD_{P}\cap D)D_{P},$ $%
yD_{P}=(yD_{P}\cap D)D_{P}$ comparable for each pair $x,y$ of nonzero non
units of $D_{P}.$ That is, $D_{P}$ is a valuation domain. Conversely if $D$
is a $\ast $-IRKT, then using Theorem \ref{Theorem N} we can establish that
for every nonzero non unit $x$ of $D$, $xD$ expressible as a $\ast $-product
of finitely many $\ast $-homog ideals of the form $xD_{P}\cap D.$ But as a $%
\ast $-IRKT is a P$\ast $MD, every $\ast $ ideal of finite type is $\ast $%
-invertible, so every $\ast $-ideal $F$ of finite type containing $%
xD_{P}\cap D$ is $\ast $-invertible making $xD_{P}\cap D$ a $\ast $-super
homog ideal. In other words we have the following result.

\begin{proposition}
\label{Proposition Q3}The following are equivalent for an integral domain $%
D: $ (1) $D$ is a $\ast $-super SH domain, i.e., every nonzero non unit $x$
of the domain $D,$ $xD$ is expressible as a $\ast $-product of finitely many 
$\ast $-super homog ideals (2) $D$ is a $\ast $-IRKT).
\end{proposition}

Note that in a $\ast $-IRKT every maximal $\ast $-ideal $M$ contains at
least one $\ast $-super homog ideal and so must be spawned by a $\ast $%
-super homog ideal. So, for $\ast =d,$ $d$-IRKT is a $\ast $-IRKT in which
every maximal $\ast $-ideal is a maximal ideal. Now, by Proposition \ref%
{Proposition Q3}, in a $\ast $-IRKT $D,$ we have $D_{P}$ a valuation domain
for every maximal $\ast $-ideal $P.$ So a $d$-IRKT is a Prufer domain.

Recall that a domain $D$ is called a generalized Krull domain (GKD) if there
is a family $\mathcal{F}$ of height one primes such that (1) $D=\dbigcap
\limits_{P\in\mathcal{F}}D_{P}$ where the intersection is locally finite and 
$D_{P}$ is a valuation ring for each $P\in\mathcal{F}.$ Indeed a GKD is an
IRKT. Following the pattern we can say that a $\ast$-IRKT whose maximal $%
\ast $-ideals are of height one is a $\ast$-GKD. Indeed a $d$-GKD is a
Prufer domain

\begin{definition}
\label{Definition Q4}Call a $\ast $-super homog ideal $I$ a $\ast $-super
homog ideal of type $1$, if $I$ is also a $\ast $-homog ideal of type $1.$
\end{definition}

Indeed as the $\ast $-product of two $\ast $-homog ideals of type $1$ is $%
\ast $-homog of type $1,$ the $\ast $-product of two $\ast $-super homog
ideals of type $1$ is a $\ast $-super homog ideal of type $1$ and the theory
runs along lines parallel to the theory based on $\ast $-homog ideals of
type $1.$

\begin{definition}
\label{Definition Q4a}Call a domain $D$ a $\ast $-super SH domain of type $1$
if for every nonzero non unit $x$ of $D$ the principal ideal $xD$ is
expressible as a $\ast $-product of $\ast $-super homog ideals of type $1.$
\end{definition}

\begin{proposition}
\label{Proposition Q5} The following are equivalent for an integral domain $%
D:$ (1) For every nonzero non unit $x$ of the domain $D,$ $xD$ is
expressible as a $\ast $-product of finitely many $\ast $-super homog ideals
of type $1,$i.e. $D$ is a $\ast $-super SH domain of type $1$ (2) $D$ is a $%
\ast $-GKD.
\end{proposition}

\begin{proof}
By Proposition \ref{Proposition Q3}, $D$ is a $\ast $-super SH domain ($\ast 
$-IRKT) and by Theorem \ref{Theorem P2} $D$ is $\ast $-WKD. Thus $D$ is a $%
\ast $-GKD. The converse can be proved in the same manner as the converse of
Proposition \ref{Proposition Q3} was. That is by assuming that $D$ is a $%
\ast $-GKD, then using Theorem \ref{Theorem N} we can establish that for
every nonzero non unit $x$ of $D$, $xD$ expressible as a $\ast $-product of
finitely many $\ast $-homog ideals of the form $xD_{P}\cap D.$ But as $D_{P}$
is a rank one valuation domain for each maximal $\ast $-ideal $P$, every
finite type $\ast $-ideal $I$ containing $xD_{P}\cap D$ would have to be $%
\ast $-invertible because a $\ast $-GKD is a P$\ast $MD, so $xD_{P}\cap D$
is a $\ast $-super homog ideal.
\end{proof}

\begin{proposition}
\bigskip \label{Proposition Q6} For each pair $a,b$ of nonzero non units in
a $\ast $-IRKT (i.e. a $\ast $-super SH domain) $D,$ $(a,b)^{\ast }=D$ or $%
(a,b)^{\ast }=$ a $\ast $-product of $\ast $-super homog ideals $I$, each
containing both $a$ and $b$ such that $(a,b)D_{M(I)}=aD_{M(I)}$ or $%
(a,b)D_{M(I)}=bD_{M(I)}.$
\end{proposition}

\begin{proof}
Because $a,b$ are nonzero non units of a $\ast $-IRKT $D,$ we can write $%
(a)=(I_{1}I_{2}....I_{l})^{\ast },$ where $I_{i}$ are mutually $\ast $%
-comaximal $\ast $-super homog ideals and similarly we can write $%
(b)=(J_{1}J_{2}....J_{m})^{\ast },$ where $m\geq n$ and $J_{j}$ are mutually 
$\ast $-comaximal $\ast $-super homog ideals. If for some $i,$ $I_{i}$ is
similar to some $J_{j},$ then $(I_{i},J_{j})$ is a unique pair in that $%
I_{i} $ is $\ast $-comaximal with each of the other $\ast $-super homog
ideals appearing in the expression for $(a)$ above and similarly for $J_{j}.$
We conclude that $(I_{i},b)^{\ast }\neq D$ and similarly $(J_{j},a)^{\ast
}\neq D$ and that there are exactly the same number of $I_{i}$s that have $%
(I_{i},b)^{\ast }\neq D$ as $J_{j}$s that have $(J_{j},a)^{\ast }\neq D.$
Suppose that, by a relabeling if necessary, $I_{1},I_{2},...,I_{r}$ are all
the $I_{i}$ such that $(I_{i},b)^{\ast }\neq D,$ and $J_{1},J_{2},...,J_{r}$
are all the $J_{j}$ such that $(J_{j},a)^{\ast }\neq D$ Thus $%
(a)=(I_{1}I_{2}...I_{r}I_{r+1}...I_{l})^{\ast }$ such that each of the $%
I_{1},...,I_{r}$ is similar to some, and hence exactly one, of the $J_{j}.$
Similarly we can write, relabeling if necessary, $%
(b)=(J_{1}J_{2}...J_{r}...J_{m})^{\ast }$ such that each of the $J_{i}$ is
similar to $I_{i}$ for $i=1,2,...,r.$ Now as $I_{i},J_{i}$ are similar, i.e. 
$M(I_{1})=M(J_{1}),$ we conclude that $(I_{i},J_{i})^{\ast }=K_{i}^{\ast },$
for $i=1,...,r,$ where $K_{i}^{\ast }=J_{i}$ if $I_{i}\subseteq J_{i}$ and $%
K_{i}^{\ast }=I_{i}$ if $J_{i}\subseteq I_{i}.$ Obviously, in either case, $%
(a,b)^{\ast }\subseteq K_{i}^{\ast }$ and as $K_{i}$ are mutually $\ast $%
-comaximal $(a,b)^{\ast }\subseteq (K_{1}K_{2}...K_{r})^{\ast }=H.$ As $%
K_{i}^{\ast }$ are the only $\ast $-super homog ideals containing both $a$
and $b$ we conclude that $(a,b)^{\ast }=(K_{1}K_{2}...K_{r})^{\ast }.$ This
can be seen as follows: Since $(a,b)^{\ast }$ $=$ $%
((K_{1}...K_{r})((K_{1}^{-1}...K_{r}^{-1})a,(K_{1}^{-1}...K_{r}^{-1})b))^{%
\ast }.$ This is because $K_{i}$ are $\ast $-invertible and $(a,b)^{\ast
}\subseteq (K_{1}K_{2}...K_{r})^{\ast }$. So, $(a,b)^{\ast }=$ $%
((K_{1}...K_{r})^{\ast }$ $(K_{1}^{-1}...K_{r}^{-1}I_{1}I_{2}$ $~...$ $%
I_{r}I_{r+1}...I_{l})^{\ast },$ $(K_{1}^{-1}...K_{r}^{-1}b))^{\ast }$ $=$ $%
((K_{1}...K_{r})^{\ast }$ $((K_{1}^{-1}I_{1})^{\ast
}...(K_{r}^{-1}I_{r})^{\ast }I_{r+1}$ $...$ $I_{l},$ $(K_{1}^{-1}J_{1})^{%
\ast }...$ $(K_{r}^{-1}J_{r})^{\ast }J_{r+1}...J_{m})^{\ast })^{\ast }.$ Now
note that $(K_{i}^{-1}I_{i})^{\ast }=D$ or $(K_{i}^{-1}I_{i})^{\ast }$ is a $%
\ast $-super homog ideal similar to $I_{i}$. Similarly $(K_{i}^{-1}J_{i})^{%
\ast }=D$ or $(K_{i}^{-1}J_{i})^{\ast }$ is a $\ast $-super homog ideal
similar to $J_{i}$. Moreover, in both cases, $((K_{i}^{-1}I_{i})^{\ast
},(K_{i}^{-1}J_{i})^{\ast })^{\ast }=D$ and, for $i\neq j,$ $%
((K_{i}^{-1}I_{i})^{\ast },(K_{j}^{-1}J_{j})^{\ast })^{\ast }$ $=D,$ $%
i,j=1,...,r.$ Moreover $((K_{i}^{-1}I_{i})^{\ast },J_{t})^{\ast }=D,$ for $%
i=1,...,r,t=r+1,...,m$ anyway and already for each $%
s=r+1,...,l_{t}=r+1,...,m $ $(I_{s},J_{t})^{\ast }=D.$ So, each of the
factors in $(K_{1}^{-1}I_{1})^{\ast }...$ $(K_{r}^{-1}I_{r})^{\ast
}I_{r+1}...I_{l}$ $=(K_{1}^{-1}...K_{r}^{-1})a$ is $\ast $-comaximal with
each of the factors in the $\ast $-product of $(K_{1}^{-1}J_{1})^{\ast }.$ $%
..(K_{r}^{-1}J_{r})^{\ast }J_{r+1}...J_{m}=$ $(K_{1}^{-1}...K_{r}^{-1})b.$
Thus $((K_{1}^{-1}...K_{r}^{-1})a,$ $(K_{1}^{-1}...K_{r}^{-1})b)^{\ast }=D$
and that gives $(a,b)^{\ast }=$ $((K_{1}...K_{r})((K_{1}^{-1}...K_{r}^{-1})a$
$(K_{1}^{-1}...K_{r}^{-1})b))^{\ast }=(K_{1}...K_{r})^{\ast }.$
\end{proof}

The above, ab-initio, proof was to stress the idea that there is a kind of
GCD at work. Below we provide an alternate statement that seems to get
similar results in a different way.

\begin{corollary}
\label{Corollary Q7} Given two nonzero elements $a,b$ in a $\ast$-super SH
domain $D,$ the following hold: (1) if there is no maximal $\ast$-ideal $P$
of $D$ that contains both $a$ and $b$, then $(a,b)^{\ast}=D,$ (2) if either
of $a,b$ is a unit, then $(a,b)^{\ast}=D,$ (3) if $P$ is a maximal $\ast$%
-ideal containing both $a,b$ then $(a,b)^{\ast}D_{P}\cap D=aD_{P}\cap D$ if $%
a|b$ in $D_{P}$ and $(a,b)^{\ast}D_{P}\cap D=bD_{P}\cap D$ if $b|a$ in $%
D_{P},$ (4) if $P_{1},P_{2},...,P_{n}$ are all the maximal $\ast$-ideals of $%
D$ that contain both $a$ and $b$ and if $I_{i}=(a,b)^{\ast}D_{P_{i}}\cap D$
for $i=1,...,n$, then $(a,b)^{\ast}=(I_{1}I_{2}...I_{n})^{\ast},$ (5) for
every pair $a,b\in D\backslash\{0\},(a,b)$ is $\ast$-invertible and so, $D$
is a P$\ast$MD and (6) a $d$-super SH domain is a Prufer domain.
\end{corollary}

\begin{proof}
(1) and (2) are straight forward. For (3) let $P$ be a maximal $\ast $-ideal
containing both $a,b.$ Then as $(a,b)^{\ast }=(a,b)^{\ast _{w}}$ we have $%
(a,b)^{\ast }D_{P}\cap D=(a,b)^{\ast _{w}}D_{P}\cap D=(a,b)D_{P}\cap D.$ Now
as $D_{P}$ is a valuation domain $a|b$ or $b|a$ in $D_{P}$ and so $%
(a,b)D_{P}=aD_{P}$ if $a|b$ and $(a,b)D_{P}=bD_{P}$ if $b|a$ in $D_{P}.$ For
(4) note that for any maximal $\ast $-ideal $P$ such that $P$ does not
contain at least one of $a$ or $b,$ $(a,b)D_{P}=D_{P}.$ Now $(a,b)^{\ast
_{w}}=\dbigcap\limits_{P\in t\text{-}Max(D)}(a,b)D_{P}=(\dbigcap%
\limits_{i=1}^{i=n}(a,b)D_{P_{i}})\cap D$ $=\dbigcap%
\limits_{i=1}^{i=n}(a,b)D_{P_{i}}\cap D)=$ $\dbigcap\limits_{i=1}^{i=n}I_{i}$
$=(I_{1}I_{2}...I_{n})^{\ast _{w}},$ $I_{i}$ being $\ast $-super homog. So $%
(a,b)^{\ast _{w}}=(I_{1}I_{2}...I_{n})^{\ast _{w}}$ and applying $\ast $ to
both sides we get $(a,b)^{\ast }=(I_{1}I_{2}...I_{n})^{\ast }.$ Finally, for
(5) and (6), note that as $(a,b)D_{P}$ is principal for each maximal $\ast $%
-ideal $P,$ because $D_{P}$ is a valuation domain, we conclude that in f $D$
is a $\ast $-super SH domain and if, for $a,b\in D$ with $\left( a,b\right)
\neq (0)$ then $(a,b)$ is $\ast _{w}$-invertible and hence $\ast $%
-invertible. This makes the $\ast $-super SHD $D$ a P$\ast $MD and $d$-super
SHD a P$d$MD which is Prufer.
\end{proof}

Part (5) of Corollary \ref{Corollary Q7} is sort of tongue in the cheek in
that for every maximal $\ast $-ideal $M$ of a $\ast $-IRKT $D$ we have that $%
D_{M}$ is a valuation domain, a necessary and sufficient condition for $D$
to be a P$\ast $MD.

An integral domain $D$ is called an almost GCD (AGCD)\ domain if for every
pair $a,b\in D\backslash\{0\}$ there is a positive integer $n$ such that $%
(a^{n})\cap(b^{n})$ is a principal ideal. Equivalently, $D$ is an AGCD
domain if (and only if) for every pair $a,b\in D\backslash\{0\}$ there is a
positive integer $n$ such that $(a^{n},b^{n})_{v}$ is a principal ideal. Now
we can write $(a^{n},b^{n})_{v}$ as $(a^{n},b^{n})_{t}$ because the number
of generators is finite. AGCD domains have been studied in \cite{Z} and in 
\cite{AZ1} as a generalization of GCD domains. Here $D$ is a GCD domain if
every pair $a,b$ of nonzero elements of $D$ has a greatest common divisor
GCD. It is well known that $D$ is a GCD domain if and only if for every pair
of nonzero elements $a,b$ the ideal $aD\cap bD$ is principal (i.e. if and
only if $(a,b)_{v}$ is principal).

Since a $\ast$-IRKT is a P$\ast$MD, and hence integrally closed $%
(a^{n},b^{n})_{t}=((a,b)^{n})_{t}.$ Also since a P$\ast$MD\ is a P$t$MD we
have $((a,b)^{n})^{\ast}=((a,b)^{n})_{t}=(a^{n},b^{n})_{t}.$

\begin{proposition}
\label{Proposition Q8} A $\ast $-IRKT $D$ is an AGCD domain if and only if
for every $\ast $-super homog ideal $A$ of $D$ we have $(A^{n})^{\ast }$
principal for some positive integer $n.$
\end{proposition}

\begin{proof}
Let $D$ be a $\ast $-IRKT. Suppose that for every $\ast $-super homog ideal $%
A$ we have $(A^{n})^{\ast }$ principal for some $n.$ Let $a,b$ be two
nonzero non units of $D.$ By Proposition \ref{Proposition Q6} we have $%
(a,b)^{\ast }=(J_{1}J_{2}...J_{r})^{\ast }$ where $J_{i}$ are mutually $\ast 
$-comaximal $\ast $-super homog ideals, each dividing out $a$ or $b$. Now
let $n_{i}$ be the positive integers such that $(J_{i})^{n_{i}}=(d_{i}).$
Let $m=LCM(\{n_{i}\}).$ Then $((a,b)^{m})^{\ast
}=(J_{1}^{m}J_{2}^{m}...J_{r}^{m})^{\ast },$ as $(J_{i}^{m})^{\ast
}=((J_{i}^{n_{i}})^{m/n_{i}})^{\ast }=$ $((d_{i})^{m/n_{i}})=(D_{i}),$ say.
But then $((a,b)^{m})^{\ast }=(D_{1}...D_{r})$ a principal ideal. Applying
the $v$-operation to both sides we have $((a,b)^{m})_{v}=(D_{1}...D_{r})$ as
a $\ast $-IRKT is integrally closed, we have $%
((a,b)^{m})_{v}=(a^{m},b^{m})_{v}.$ Now $(a^{m},b^{m})_{v}$ being principal
leads to $a^{m}D\cap b^{m}D$ Indeed if one of $a$ or $b$ is a unit, or if $%
a,b$ are $\ast $-comaximal, $(a,b)^{\ast }$ is principal which leads to $%
aD\cap bD$ principal and so $D$ is an AGCD domain. Conversely, let the $\ast 
$-IRKT $D$ be an AGCD domain. Then, as every $\ast $-homog ideal $I$ is such
that $I^{\ast }$ is of finite type and as $D$ is a $\ast $-IRKT, we have $%
(I^{n})^{\ast }$ principal for some $n$ \cite[Theorem 3.9]{Z}.
\end{proof}

\begin{definition}
\label{Definition R1} A $\ast$-homog ideal $I$ will be called a $\ast$%
-almost factorial homog ($\ast$-af-homog) ideal if for each $\ast$- ideal $J$
of finite type containing $I$ we have $(J^{n})^{\ast}$ principal, for some
positive integer $n$. Also, a domain $D$ will be called $\ast$-af-SH domain
if for every nonzero non unit $x$ of $D,$ $xD$ is expressible as a $\ast$%
-product of finitely many $\ast$-af-homog ideals.
\end{definition}

Indeed a $\ast$-af-homog ideal $I$ is $\ast$-super homog, as $(J^{n})^{\ast}$
principal implies $J$ is $\ast$-invertible for each $\ast$- ideal $J$ of
finite type containing $I.$

\begin{proposition}
\label{Proposition R1a} The $\ast $-product of a finite number of $\ast $%
-af-homog ideals is uniquely expressible as a $\ast $-product of mutually $%
\ast $-comaximal $\ast $-af-homog ideals.
\end{proposition}

\begin{proof}
We first show that the $\ast $-product of two similar $\ast $-af-homog
ideals $K,L$ is a $\ast $-af-homog ideal similar to $I,J$. For this let $J$
be a $\ast $-ideal of finite type containing $KL$, i.e. $J\supseteq KL.$
Then as $K,L$ are both $\ast $-af-homog and hence $\ast $-super homog we
have $(JL^{-1})^{\ast }$ a $\ast $-ideal of finite type containing $K.$ But
then, by definition, there a positive integer $m$ such that $%
((JL^{-1})^{m})^{\ast }=dD,$ or $(J^{m})^{\ast }=(L^{m})^{\ast }dD.$ Now as,
for some positive integer $n,$ we have $(L^{n})^{\ast }$ is principal we
conclude that $(J^{mn})^{\ast }$ is principal. Now $(KL)^{\ast }$ is $\ast $%
-homog similar to $K$ and $L$ because $K$ and $L$ are similar and because
for each $\ast $-ideal $J$ of finite type containing $(KL)^{\ast }$ there is
a positive integer $r$ such that $(J^{r})^{\ast }$ is principal we conclude
that $(KL)^{\ast }$ is indeed a $\ast $-af-homog ideal, similar to $K$ and $%
L.$ That a $\ast $-product of finitely many $\ast $-af-homog ideals similar
to each other is a $\ast $-af-homog ideal similar to them can be shown by
doing it taken two at a time. Next $\ast $-af-homog ideals being $\ast $%
-homog we can express the $\ast $-product uniquely as a $\ast $-product of
mutually $\ast $-comaximal $\ast $-homog ideals obtained by taking $\ast $%
-products of similar $\ast $-homog ideals. Now in this case the $\ast $%
-products of those mutually similar $\ast $-homog ideals are $\ast $%
-af-homog ideals, as we found in the proof of Proposition \ref{Proposition
Q3}, by noting that if $xD$ is a $\ast $-product of $\ast $-af-homog ideals
then $xD\cap D$ is one of those $\ast $-af-homog ideals.
\end{proof}

\begin{theorem}
\label{Theorem R2} The following are equivalent for an integral domain $D:$
(1) for every nonzero non unit $x$ of the domain $D,$ $xD$ is expressible as
a $\ast$-product of finitely many $\ast$-af-homog ideals (2) $D$ is an AGCD $%
\ast$-IRKT.
\end{theorem}

\begin{proof}
$D$ is a $\ast $-IRKT by Proposition \ref{Proposition Q3} and by Proposition %
\ref{Proposition Q8}, supported by the definition of $\ast $-af-ideals, $D$
is an AGCD $\ast $-IRKT. For the converse note that, as we have already
observed, every principal ideal $xD$ generated by a nonzero non unit $x$ can
be expressed as a $\ast $-product of $\ast $-homog ideals, each of which is, 
$\ast $-invertible and, expressible as $xD_{P}\cap D$ where $D_{P}$ is a
valuation domain. Now a $\ast $-ideal $J$ of finite type containing $%
xD_{P}\cap D$ is a $\ast $-ideal of finite type of an AGCD domain in which $%
\ast =t$ and so there must be a positive integer $n$ such that $%
(J^{n})^{\ast }$ is principal. Thus each of $xD_{P}\cap D$ is $\ast $%
-af-homog.
\end{proof}

We can define $\ast$-af-homog ideals of type $1$ and type $2$ and prove
obvious results about AGCD $\ast$-GKD and AGCD $\ast$-Krull.

Call a $\ast $-homog ideal $I$ a $\ast $-af-homog ideal of type $1,$ if $I$
is a $\ast $-af-homog ideal and a $\ast $-homog ideal of type $1.$ Now let $%
I,J$ be two $\ast $-af-homog ideals of type $1,$ Then $IJ$ is a $\ast $%
-af-homog ideal, by the proof of Proposition \ref{Proposition R1a} and of
type $1$ by the remark before Theorem \ref{Theorem P2}

As we have done in other cases let's call an integral domain $D$ a $\ast $%
-af-SH domain of type $1$ if for every nonzero non unit $x\in D,$ $xD$ is
expressible as a $\ast$-product of $\ast$-af-homog ideals of type $1.$

Indeed a $\ast $-af-homog ideal of type $1$ is a $\ast $-super homog ideal
of type $1$ and so a $\ast $-af-SH domain of type $1$ is at least a $\ast $%
-GKD. Next, an AGCD $\ast $-GKD is at least a $\ast $-GKD. So the proof of
the following statement will run along lines similar to that of Proposition %
\ref{Proposition Q5}.

\begin{proposition}
\label{Proposition R3} The following are equivalent for an integral domain $%
D:$ (1) $D$ is a $\ast$-af- domain of type $1$, i.e. for every nonzero non
unit $x$ of the domain $D,$ $xD$ is expressible as a $\ast$-product of
finitely many $\ast$-af-homog ideals of type $1,$ (2) $D$ is an AGCD $\ast$%
-GKD.
\end{proposition}

Next call a $\ast $-homog ideal $I,$ a $\ast $-af-homog ideal of type $2$ if 
$I$ is $\ast $-af-homog and every $\ast $-homog ideal $J$ containing $I$ is $%
\ast $-af-homog such that $J^{\ast }=(M(J)^{n})^{\ast }$ for some positive
integer $n.$ It is easy to see that a $\ast $-af-homog ideal of type $2$ is
a $\ast $-af-homog of type $1.$ Now we can, indeed, call $D$ a $\ast $-af-SH
domain of type $2$ if for every nonzero non unit $x\in D$ the ideal $xD$ is
expressible as a $\ast $-product of finitely many $\ast $-af-homog ideals of
type $2.$ Thus a $\ast $-af-SH domain of type $2$ is a $\ast $-af-SH domain
of type $1$ and so a $\ast $-GKD.

\begin{theorem}
\label{Theorem R4} Let $D$ be an integral domain and suppose that every
nonzero proper principal ideal of $D$ is expressible as a finite $\ast $%
-product of $\ast$-af-homog ideals of type $2$. Then $D$ is an AGCD $\ast $%
-Krull domain. Conversely if $D$ is an AGCD $\ast$-Krull domain, then every
nonzero proper principal ideal of $D$ is expressible as a $\ast$-product of
finitely many $\ast$-af-homog ideals of type $2$.
\end{theorem}

The proof should be somewhat simpler than that of Theorem \ref{Theorem P4}
because we have assumed $I$ $\ast$-af-homog and that makes $M(I)$ $\ast $%
-invertible for each $\ast$-af-homog ideal $I,$ making $D$ a $\ast$-Krull
domain right away.

The AGCD $\ast $-Krull domains were first studied by U. Storch in \cite{S}.
The easiest to access these domains is taking Dedekind domains with torsion
class groups.

\begin{definition}
\label{Definition S} Call a $\ast$-homog ideal $\ast$-factorial homog ($\ast 
$-f-homog) if every $\ast$-ideal of finite type containing $I$ is principal$%
. $
\end{definition}

In other words, repeating Definition \ref{Definition Q}, a nonzero $\ast $%
-ideal of finite type is called $\ast $-f-homog if (1)S for each $\ast $%
-ideal of finite type containing $I$ is principal and (2)S For every pair of
proper integral $\ast $-ideals $A,B$ of finite type containing $I,$ $%
(A+B)^{\ast }\neq D.$ Indeed a $\ast $-f-homog ideal $I$ is $\ast $-super
homog and so has all the properties listed in Proposition \ref{Proposition
Q1}. In particular as a $\ast $-f-homog ideal is principal, we can use
\textquotedblleft $\ast $-f-homog element $x$\textquotedblright\ instead of
\textquotedblleft $\ast $-f-homog ideal $xD$\textquotedblright .
Consequently, we can say that, the set of all factors of $\ast $-f-homog
elements is totally ordered under inclusion of the principal ideals
generated by them, i.e. $\ast $-f-homog element is a rigid element. To be
exact we have the following result linking \textquotedblleft rigid
element\textquotedblright\ with \textquotedblleft $\ast $-f-homog
element\textquotedblright .

\begin{proposition}
\label{Proposition S1} For the generator of the principal ideal $rD$ the
following are equivalent. (1) $rD$ is $\ast$-f-homog, (2) $r$ is a rigid
element that belongs to a unique maximal $t$-ideal and every $\ast$-ideal of
finite type containing $r$ is principal.
\end{proposition}

Proof. (1) $\Rightarrow$ (2) Suppose $rD$ is $\ast$-f-homog then, as already
mentioned, $r$ is rigid i.e. for each pair $x,y$ of factors of $r,$ $x|y$ or 
$y|x.$ The condition (2)S ensures that $r$ belongs to a unique maximal $\ast 
$-ideal the remainder is taken care of by condition (1)S. For (2) $%
\Rightarrow$ (1) Let $I$ be a $\ast$-ideal of finite type containing $rD.$
Then $I$ contains $r$ and hence must be principal, this takes care of (1)S.
Next let $A$ and $B$ be two integral $\ast$-ideals of finite type containing 
$rD.$ Then we have seen that $A=xD$ and $B=yD$ for some factors $x $ and $y$
of $r.$ But as $x|y$ or $y|x$ we have $(A+B)^{\ast}\neq D$ which is (2)S.

It may be noted that merely saying \textquotedblleft$r$ is a rigid element
belonging to a maximal $\ast$-ideal $P$\textquotedblright, is not enough. We
need to make sure that every finite type proper $\ast$-ideal containing $r$
is principal. This is because in the Dedekind domain $%
\mathbb{Z}
\lbrack\sqrt{-5}],$ where, of course, $\ast=d.$ For the prime ideal $P=(2,1+%
\sqrt{-5})$ we have $P^{2}=(2)$ where $2$ is irreducible in $%
\mathbb{Z}
\lbrack\sqrt{-5}]$ and so fits the definition of a rigid element, but $2$ is
not $\ast$-f-homog because $P$ contains $2$ yet $P$ is not principal.

Note that the $\ast$-product $(IJ)^{\ast}$ of two similar $\ast$-f-homog
ideals $I,J$ is a $\ast$-f-homog ideal similar to both $I$ and $J.$ ($I,J$
are similar $\ast$-super homog, so $(IJ)^{\ast}$ is $\ast$-super homog,
similar to $I$ and $J$ and ($(IJ)^{\ast}\subseteq I\subseteq J$ or $%
(IJ)^{\ast}\subseteq J\subseteq I),$ say $I\subseteq J.$ Now let $C$ be a
star ideal of finite type containing $(IJ)^{\ast}$. Since $I$ is principal
so is $I^{-1}$ and so $CI^{-1}$ is a $\ast$-ideal of finite type and $%
CI^{-1}\supseteq(IJ)^{\ast }I^{-1}=J.$ So $CI^{-1}$ is principal which
forces $C$ to be principal.) Consequently a product of finitely many $\ast$%
-f-homog ideals/elements is expressible, uniquely, up to associates and
order, as a product of mutually $\ast$-comaximal $\ast$-f-homog
ideals/elements.

\begin{theorem}
\label{Theorem S2} Suppose that every nonzero non unit $x$ of $D$ generates $%
xD$ that is expressible as a $\ast $-product of finitely many $\ast $%
-f-homog ideals. Then $D$ is a GCD $\ast $-IRKT whose nonzero non units are
uniquely expressible as products of mutually co prime $\ast $-f-homog
elements. Conversely if $D$ is a GCD $\ast $-IRKT then every proper
principal ideal of $D$ is expressible as a finite $\ast $-product of $\ast $%
-f-homog ideals.
\end{theorem}

\begin{proof}
Because every $\ast $-f-homog ideal is $\ast $-super homog ideal, $D$ is a $%
\ast $-IRKT by Proposition \ref{Proposition Q3}. It is also well known that
if $D$ is a $\ast $-IRKT then $D$ is a P$\ast $MD and so $\ast =t$. We have
already established, in Proposition \ref{Proposition K}, that if $%
xD=(I_{1}I_{2}...I_{n})^{\ast }$, where $I_{i}$ are $\ast $-homog ideals
then $xD=(J_{1}J_{2}...J_{r})^{\ast }$ where $J_{j}$ are mutually $\ast $%
-comaximal $\ast $-homog ideals and this expression is unique up to order
etc. Indeed as a $\ast $-f-homog ideal is $\ast $-homog the statement holds
here too. Now let $a,b$ be two nonzero elements of $D$. We can assume that $%
aD=(A_{1}...A_{r})^{\ast }$ where $A_{i}$ are mutually $\ast $-comaximal $%
\ast $-f-homog ideals. $(A_{1}^{\ast }...A_{r}^{\ast })^{\ast
}=a_{1}D...a_{r}D,$ here $A_{i}^{\ast }=a_{i}D$ because $A_{i}$ is $\ast $%
-f-homog . Similarly, $bD=b_{1}D...b_{s}D$ where $b_{i}$ are mutually $\ast $%
-comaximal. Let $P_{1},P_{2}...P_{n}$ be all the maximal $\ast $-ideals that
contain both $a$ and $b.$ By rearranging we can assume that $%
a_{i}D,b_{i}D\in P_{i},$ for $i=1,2,...n.$ Now as $a_{i}D,b_{i}D$ are $\ast $%
-f-homog belonging to the same maximal $\ast $-ideal $P_{i}$ $a_{i}|b_{i}$
or $b_{i}|a_{i}.$ Thus $(a_{i},b_{i})^{\ast }=a_{i}D$ or $b_{i}D$ according
as $a_{i}|b_{i}$ or $b_{i}|a_{i}.$ Let's denote $(a_{i},b_{i})^{\ast }$ by $%
d_{i}D.$ Proceeding as in Proposition \ref{Proposition Q6} we have $%
(a,b)^{\ast }=(\dprod\limits_{i=1}^{n}(a_{i},b_{i})^{\ast })^{\ast }=(\dprod
d_{i}D)^{\ast }=(d_{1}d_{2}...d_{n}D)^{\ast }=d_{1}d_{2}...d_{n}D.$ Thus $%
(a,b)^{\ast }=d_{1}d_{2}...d_{n}D$ a principal ideal. Applying the $v$%
-operation on both sides we get $(a,b)_{v}=d_{1}d_{2}...d_{n}D$ a principal
ideal. As $a,b$ were arbitrary we conclude that $D$ is a GCD-domain. Also as
we have already established that $D$ is a $\ast $-IRKT, we are done. For the
converse note that in a $\ast $-IRKT \thinspace $\ast =t,d$ a $t$-IRKT is an
IRKT and a $d$-IRKT is a Prufer domain. That a GCD IRKT is semirigid (every
nonzero non unit expressible as a product of finitely many rigid elements)
was established in \cite{ZR3}, where IRKT was dubbed as IKT domain or use
the following lemma.
\end{proof}

\begin{lemma}
\label{Lemma S2a} The following are equivalent for a $\ast $-ideal $I$ of
finite type in a GCD domain $D:$ (1) $I$ is $\ast $-homog, (2) $I$ is $\ast $%
-super homog, (3) $I$ is $\ast $-f-homog and (4) $I=rD$ where $r$ is a rigid
element.
\end{lemma}

\begin{proof}
Note that in a GCD domain $\ast =t.$ Now (1) $\Rightarrow $ (3) because $I$
is $\ast $-homog such that every $\ast $-ideal of finite type of $D$ is
principal, being $t$-ideal of finite type of a GCD domain and one that fits
the definition of a $\ast $-f-homog ideal (3) $\Rightarrow $ (2) because
principal is invertible and (2) $\Rightarrow $ (1) is obvious. Now, in a GCD
domain a rigid element $r$ belongs to a unique maximal $t$-ideal $M=\{x,$ $%
(r,x)_{v}\neq D\}.$ This is because $(r,x)_{v}\neq D$ implies that $r$ has a
non unit common factor $r_{x}$with $x$ So for each $x\in M\backslash \{0\}$
we have $x=r_{x}(x/r_{x})$ where $r_{x}$ is a non unit factor of $r.$ Now
let $x_{1},...,x_{n}\in M\backslash \{0\}.$ Then $%
x_{i}=r_{x_{i}}(x/r_{x_{i}})$ where $r_{x_{i}}$ are non unit factors of the
rigid element $r.$ Since for all $a,b|r$ we have $a|b$ or $b|a$ we have $%
(x_{1},x_{2},...,x_{n})\subseteq (r_{x_{j}})$ which means $%
(x_{1},x_{2},...,x_{n})_{v}\subseteq (r_{x_{j}})$ for some $1\leq j\leq n.$
But as $(r_{x_{j}},x)_{v}=(r_{x_{j}})\neq D$ we have $%
(x_{1},x_{2},...,x_{n})_{v}\subseteq (r_{x_{j}})\subseteq M$ and $M$ is a $t$%
-ideal and there is no $t$-ideal not contained in $M$ that contains $r.$ For
if $N$ were such a $t$-ideal, then there is say $\alpha \in N\backslash M.$
But then $(\alpha ,r)_{v}=D.$ Whence any $t$-ideal containing $r$ must be
contained in $M$ . Thus $I=rD$ is $\ast $-homog and (4) $\Rightarrow $ (1).
Now a $\ast $-f-homog ideal that is principal must be a generated by a rigid
element by Proposition \ref{Proposition S1} and this establishes (3) $%
\Rightarrow $ (4).
\end{proof}

We can call a $\ast $-f-homog ideal $I$ $=xD$ a $\ast $-f-homog ideal of
type $1$ if, in addition, 1$I$ is $\ast $-homog of type $1$. Indeed $I$ $=xD$
a $\ast $-f-homog ideal is of type $1$ if and only if for every $\ast $%
-f-homog ideal $A=yD$ containing $xD,$ i.e. $y|x$ in $D,$ there is a
positive integer $n$ such that $x|y^{n}$. If we develop a theory of
factorization on it we will get a theorem like the following.

\begin{theorem}
\label{Theorem S3} The following are equivalent for an integral domain $D:$
(1) For every nonzero non unit $x$ of the domain $D,$ $xD$ is expressible as
a $\ast$-product of finitely many $\ast$-f-homog ideals of type $1$ (2) $D$
is a GCD $\ast$-GKD.
\end{theorem}

\begin{proof}
By (1), using Theorem \ref{Theorem S2}, $D$ is a GCD $\ast $-IRKT, because
every $\ast $-f-homog ideal of type $1$ is a $\ast $-f-homog ideal. But a $%
\ast $-f-homog ideal of type $1$ is also a $\ast $-super homog ideal of type 
$1$ and so Proposition \ref{Proposition Q5} applies to give that $D$ is a
GCD $\ast $-GKD. For the converse the reader may refer to \cite{AAZ1} or
just note that a $\ast $-GKD is a $\ast $-IRKT whose maximal $\ast $-ideals
are of height $1.$
\end{proof}

The domains of Theorem \ref{Theorem S3} were studied in \cite{AAZ1} under
the name of Generalized Unique Factorization Domains (GUFDs).

We can call a $\ast $-f-homog ideal $xD$ of type 2 if $xD=(M(xD))^{n}$ and
get a theory of UFD's. Of course that is too well known to repeat here.

\section{Restricted or weak theories}

Before we get down to explaining the restricted theories let us take care of
a topic that is in a way essential to them. The topic is that of (integral) $%
\ast $-invertible $\ast $-ideals. It is often noted that an integral
invertible ideal behaves like a principal ideal in many respects, for
example an invertible ideal is locally principal. In fact a nonzero finitely
generated ideal $I$ is invertible if and only if $I$ is locally principal,
i.e., $ID_{M}$ is principal for every maximal ideal $M.$ In a similar manner
a $\ast $-invertible $\ast $-ideal may be characterized by,
\textquotedblleft a $\ast $-ideal of finite type $I$ such that $ID_{P}$ is
principal for each maximal $\ast $-ideal $P$ of $D$\textquotedblright\ (see
We plan to use this feature in the following in a somewhat indirect manner.
But first we must talk about another important property of integral $\ast $%
-invertible $\ast $-ideals, in the context of $\ast $-SH domains.

\begin{theorem}
\label{Theorem 4Aa} Let $D$ be a $\ast$-SH domain and $I$ a $\ast$%
-invertible $\ast$-ideal of $D.$ Then $I$ is uniquely expressible, up to
order, as a $\ast$-product of mutually $\ast$-comaximal $\ast$-homog ideals.
\end{theorem}

Proof. Indeed as $D$ is of finite $\ast $-character, $I$ is contained in at
most a finite number of maximal $\ast $-ideals $P_{1},P_{2},...,P_{n}$ we
have $I=\cap _{i=1}^{n}(ID_{P_{i}}\cap D).$ Now because $P_{i}$ shares no
nonzero prime ideal with any other maximal $t$-ideal we conclude that none
of $I_{i}=(ID_{P_{i}}\cap D)$ is contained in any maximal $\ast $-ideal
other than $P_{i}$ for $i=1,...,n.$ Next as $D$ is of finite $\ast $%
-character each of $I_{i}$ is a $\ast $-ideal of finite type. Thus each of $%
I_{i}$ is $\ast $-homog. Also $I_{i}$ are mutually $\ast $-comaximal by
Remark \ref{Remark D}. So, $I=\cap
_{i=1}^{n}I_{i}=(I_{1}I_{2}...I_{n})^{\ast _{w}}=(I_{1}I_{2}...I_{n})^{\ast
} $. That this expression is unique, up to order, follows from proofs of
similar results in earlier sections such as Proposition \ref{Proposition K}.

\begin{corollary}
\label{Corollary 4Ab} Let $D$ be a $\ast$-SH domain. Then the $\ast$-class
group of $D$ is $0$ if and only if every $\ast$-invertible $\ast$-homog
ideal of $D$ is principal.
\end{corollary}

Taking a cue from the above result we make the following definition.

\begin{definition}
\label{Definition 4Ac} An integral $\ast$-ideal $I$ is a $\ast$-weakly
factorial homogeneous ($\ast$-wf-homog) ideal if $I$ is $\ast$-homog such
that when $\ast$-invertible every $\ast$-invertible $\ast$-ideal $J$
containing $I$ is a principal ideal.
\end{definition}

So a $\ast $-homog ideal $I$ is $\ast $-wf-homog if $I$ is principal along
with all the $\ast $-invertible $\ast $-ideals containing it in the event
that $I$ is $\ast $-invertible, otherwise it is just a $\ast $-homog ideal.
This is because being $\ast $-homog $I$ is of finite type and $I$ is
contained in a unique maximal $\ast $-ideal $M,$ so $I$ is $\ast $%
-invertible if and only if $ID_{M}$ is principal. We may call the generator
of a principal $\ast $-wf-homog ideal an $\ast $-wf-homog element.

It is easy to see that the $\ast $-product $(IJ)^{\ast }$ of two similar $%
\ast $-wf-homog ideals $I,J$ is a $\ast $-wf-homog ideal similar to both $I$
and $J.$ ($I,J$ are similar $\ast $-homog, so $(IJ)^{\ast }$ is $\ast $%
-homog and similar to $I$ and $J$. Also if $(IJ)^{\ast }$ is $\ast $%
-invertible then so are both of $I$ and $J$ and hence, by definition, have
the property that every $\ast $-invertible $\ast $- ideal containing each is
principal. Next let $X$ be a $\ast $-invertible $\ast $-ideal such that $%
X\supseteq IJ$. Then $XI^{-1}\supseteq J$ making $XI^{-1}$ a principal
ideal, because it is a $\ast $-invertible $\ast $-ideal that contains $J$.
Now if $XI^{-1}$ is principal, say $XI^{-1}=rD$, then, since $I$ is already
principal, $XI^{-1}$ is principal.) Consequently a product of finitely many $%
\ast $-wf-homog elements is expressible, uniquely, up to associates and
order, as a product of mutually $\ast $-comaximal $\ast $-wf-homog elements.

\begin{definition}
\label{Definition 4Ad} Call an integral domain $D$ a $\ast $- weakly
factorial SH ($\ast $-wf-SH) domain, if every nonzero non unit of $D$ is
expressible as a finite product of $\ast $-wf-homog elements.
\end{definition}

\begin{proposition}
\label{Proposition 4Ae} A $\ast$-SH domain with trivial $\ast$-class group
is a $\ast$-wf-SH domain. Conversely a $\ast$-wf-SH domain is a $\ast$-SH
domain with trivial $\ast$-class group.
\end{proposition}

\begin{proof}
Indeed if the $\ast $-class group of $D$ is zero, every $\ast $-invertible $%
\ast $-ideal of $D$ is principal. Now for every nonzero non unit $x$ in a $%
\ast $-SH domain $xD=(I_{1}I_{2}...I_{n})^{\ast },$ where $I_{i}$ are
mutually $\ast $-comaximal $\ast $-invertible $\ast $-ideals. With the added
restriction of trivial $\ast $-class group, each of $I_{i}=x_{i}D$ is a
principal ideal which fits the definition of a $\ast $-wf-homog element
(indeed every $\ast $-invertible $\ast $-ideal containing $x_{i}D$ is
principal because $Cl_{\ast }(D)=(0))$. Hence, as the $\ast $-operation is
ineffective on principal ideals, $xD=d_{1}d_{2}...d_{n}D$, or $%
x=ed_{1}d_{2}...d_{n}$ is a product of $\ast $-wf-homog elements, where $e$
is a unit. Conversely, it is obvious that (a) $D$ is of finite $\ast $%
-character and (b) every prime ideal of $D$ contains a $\ast $-wf-homog
element which generates a $\ast $-homog ideal. It is now easy to show that
every maximal $\ast $-ideal of $D$ is spawned by a $\ast $-wf-homog
principal ideal and that no two distinct maximal $\ast $-ideals contain a
nonzero prime ideal. So, a $\ast $-wf-SH domain $D$ is a $\ast $-SH domain.
Next to show that the $\ast $-class group of $D$ is trivial, take an
integral $\ast $-invertible $\ast $-ideal $I$ of $D$ and let $0\neq x\in I.$
By Theorem \ref{Theorem 4Aa}, $I=(I_{1}I_{2}...I_{n})^{\ast }$ where each of 
$I_{i}$ is a $\ast $-homog ideal and $I_{i}$ mutually $\ast $-comaximal.
Pick one, say $I_{k},$ and note that $x$ is $d_{1}d_{2}...d_{r}$ of $\ast $%
-wf-homog elements belongs to $I_{k}.$ Since for each $i,d_{i}D$ is $\ast $%
-homog and since $d_{i}$ are mutually $\ast $-comaximal, by Corollary \ref%
{Corollary H} only one of the $d_{i}$ belongs to $I_{k}.$ Now $d_{i}$ being
a $\ast $-wf-homog element, $d_{i}D$ has the property that any $\ast $%
-invertible $\ast $-ideal containing it is principal, by Definition \ref%
{Definition 4Ac}. Finally as $I$ and $I_{i}$ were arbitrary, we conclude
that every integral $\ast $-invertible $\ast $-ideal of $D$ is principal.
\end{proof}

Examples: (a) (When $\ast=d$ and no restriction on dimension). Let $(R,M)$
be a regular local domain of dimension $n\geq2$, let $L$ be the quotient
field of $R$ and let $X$ be an indeterminate over $L.$ Then the ring $%
D=R+XL[X]$ is an $n+1$ dimensional $d$-SH factorial domain such that $%
Cl_{d}(D),$ the ideal class group of $D$ is zero.

Illustration: By \cite[Cororollary 1.3]{CMZ} $D$ is a GCD domain and so $%
Cl_{t}(D)=(0)$ and as $Cl_{d}(D)\subseteq Cl_{t}(D)$ we conclude that $%
Cl_{d}(D)=(0).$ Also, by \cite[Theorem 4.21]{CMZ}, the maximal ideals of $D$
are (a) $M+XL[X]$ where $M$ is the maximal ideal of $R$ and (b) ideals of
the type $f(X)D$ where $f(X)$ is an irreducible element and hence a prime,
of $D$ such that $f(0)=1.$ Next a typical nonzero non unit $f(X)$ of $D$ can
be expressed as $\frac{a}{b}X^{r}(1+Xg(X))$ where $a,b\in D\backslash \{0\}$%
, $b=1$ if $r=0$ and $g(X)\in K[X].$ Clearly $f(X)=\frac{a}{b}X^{r}\times
(1+f_{1}(X))^{r_{1}}\times ...\times (1+f_{m}(X))^{r_{m}}$ where the $%
(1+f_{i}(X)),i=1,..,m,$ and $r$, $r_{i}\geq 0,$ are irreducible and hence
generate a principal maximal ideal each and of course $\frac{a}{b}X^{r}\in
M+XL[X]$. (Indeed if $r=0$, $\frac{a}{b}X^{r}=a$ and if $a$ is a non unit
then $a\in M+XL[X].$ Thus $f(X)$ belongs to only finitely many maximal i.e.
maximal $d$-ideals). That no two maximal ideals contain a common nonzero
prime ideal is obvious. In sum $D$ is a $d$-SH factorial domain$,$ which
works out to be an h-local domain with zero ideal class group.

Sticking with $D=R+XL[X],$ where $R$ is quasi local we note that a maximal
ideal (hence a maximal $d$-ideal) of $D$ is either $M+XL[X]$ where $M$ is
the maximal ideal of $R$ or a height one principal prime ideal of the form $%
f(X)D$ (\cite[Theorem 4.21]{CMZ}). Thus a $d$-homog ideal $I$ of $D$ such
that $I\cap R$ is non-trivial is of the form $I\cap R+XL[X].$ On the other
hand any $d$-homog ideal $J$ with $M(J)\cap R=(0)$ would have to be a power
of a prime of the form $f(X)D$ where $f(X)$ generates the maximal ideal $%
M(J).$ For, by part (a) of Proposition 4.12 of \cite{CMZ} $J=h(X)(F+XL[X])$
where $F$ is a $D$-submodule of $L$ and such that $h(0)F\subseteq D$ and $%
h(X)\in L[X].$ We claim that $h(0)\neq 0$ for otherwise $h(X)(F+XL[X])$
would be contained in $XL[X]$ and so in $M+XL[X].$ Finally as $%
J=h(X)(F+XL[X])\subseteq f(X)D$ where $f(X)D$ is a height one prime ideal,
there a positive integer $n$ such that $J\subseteq f(X)^{n}D$ but $%
J\nsubseteq f(X)^{n+1}D$ or $J/f(X)^{n}\subseteq D$ but $J/f(X)^{n}%
\nsubseteq f(X)D.$ We claim that $J/f(X)^{n}$ is not contained in any
maximal deal. For if $J/f(X)^{n}$ is contained in a maximal ideal $N$ then $%
J $ is contained in $N,$ contradicting the assumption that $J$ is $d$-homog.
Thus $J/f(X)^{n}=D,$ making $J=f(X)^{n}D.$ Thus a finite product of $d$%
-homog ideals of $D$ is an ideal of the form $l(X)(A+XL[X])$ where $l(X)$ is
a polynomial in $D$ with $l(0)=1$ and $A$ an ideal of $R.$ Also by Lemma
4.41 of \cite{CMZ} $(l(X)(A+XL[X]))_{t}=l(X)(A_{t}+XL[X]).$Thus we have
proved the following result.

\begin{lemma}
\label{Lemma 4Ae1} Let $(R,M)$ be quasi local, $L$ the quotient field of $R$%
, $X$ an indeterminate over $L$ and let $D=R+XL[X].$ Then $D$ is $d$-SH with 
$d$-homog ideals $J$ described by (a) $J=f(X)^{n}D,$ when $J\cap R=(0)$ and
(b) $J=J\cap R+XL[X],$ when $J\cap R\neq(0)$. Moreover a finite product of $%
d $-homog ideals of $D$ is of the form $J=l(X)(A+XL[X])$ and $%
J_{t}=l(X)(A_{t}+XL[X]).$ Finally, every principal ideal of $D$ is of the
form $J=al(X)D$ where $a\in D.$
\end{lemma}

The construction $D=R+XL[X]$ where $L$ is the quotient field of $R,$ will
deliver $D$ with $Cl_{d}(D)=0,$ when $Cl_{d}(R)=0.$ This is because in $%
D=R+XL[X]$ every finitely generated ideal is of the form $F=f(X)JD$ where $%
f(X)\in T$ and $J$ is a finitely generated ideal of $R$ \cite[Proposition
4.12]{CMZ}. So every invertible ideal of $D$ is principal if and only if
every invertible ideal of $R$ is principal. Thus if $R$ is quasi local then $%
D=R+XL[X]$ is a $d$-SH domain with $Cl_{d}=0.$

The above reasoning works in the $Cl_{t}(D)$ trivial or torsion cases too if
we look at it this way: If an ideal $G$ of $R+XL[X]$ is $t$-invertible then
there is a finitely generated ideal $F=f(X)JD$ of $D$ such that $%
G_{t_{D}}=(f(X)JD)_{t_{D}},$ where $t_{D}$ denotes the $t$-operation of the
domain $D.$ But by Lemma 4.41 of \cite{CMZ} $G_{t_{D}}=f(X)J_{t_{R}}D.$ So
every $t$-invertible $t$-ideal of $D$ is principal if and only if every $t$%
-invertible $t$-ideal of $R$ is principal. In other words, $Cl_{t}(D)=0$ $%
\Leftrightarrow $ $Cl_{t}(R)=0.$ We have seen that for a $t$-invertible
ideal $G\in R+XL[X]$ we have $G_{t_{D}}=f(X)J_{t_{R}}D.$ So, $%
(G^{n})_{t_{D}}=f(X)(J^{n})_{t_{R}}D.$ Thus $(G^{n})_{t_{D}}$ is principal
if and only if $(J^{n})_{t_{R}}$ is principal. Thus $Cl_{t}(D)$ is torsion $%
\Leftrightarrow $ $Cl_{t}(R)$ is torsion.

Next every $t$-local domain, i.e. a quasi local domain whose maximal ideal
is a $t$-ideal is an example of a $t$-wf-SH domain. This is because in a $t$%
-local domain $(D,M)$ every $t$-invertible ideal is invertible and hence
principal, \cite{ACZ}.

We shall see other examples as we define the $\ast$-homog ideals defining
the various clones of the $\ast$-wf-SH domains.

Example (b). Let $R$ be a $t$-local domain and let $D=R+XL[X]$ be as in
Lemma \ref{Lemma 4Ae1} then $D$ is an example of a $t$-wf-domain.

Illustration: Indeed by \ref{Lemma 4Ae1} a $t$-wf-homog ideal $J$ of $D$ is
either principal of the form $J=f(X)^{n}D,$ when $J\cap R=(0)$ or $%
J_{t}=(J\cap R)_{t}+XL[X],$ when $J\cap R\neq (0).$ Here $f(X)^{n}D$ is
principal (in fact a $t$-f-homog ideal) that satisfies the condition that if 
$t$-invertible then every $t$-invertible $t$-ideal containing it is
principal. Of course the ideal $J_{t}=(J\cap R)_{t}+XL[X]$ satisfies the
same condition because its being $t$-invertible or principal depends upon $%
(J\cap R)_{t}$ being $t$-invertible or principal which is a $t$-ideal of a $%
t $-local ring $R.$

\begin{definition}
\label{Definition 4Af} Call a $\ast$-homog ideal $I$ $\ast$-wf-homog of type 
$1$, if $I$ is a $\ast$-homog ideal of type $1$ such that whenever $I$ is a $%
\ast$- invertible $\ast$- ideal every $\ast$-invertible $\ast$-ideal $J$
containing $I$ is principal.
\end{definition}

\noindent As above we can call the generator of a principal $\ast$-wf-homog
ideal a $\ast$-wf-homog element and define a $\ast$-wf-SH domain of type $1$
as the domain whose nonzero non unit elements are expressible as products of 
$\ast$-wf-homog elements of type $1.$

\begin{proposition}
\label{Proposition 4Ag} A $\ast$-wf-SH domain of type $1$ is a $\ast$-weakly
Krull domain with trivial $\ast$-class group.
\end{proposition}

The proof follows as we chase the definitions. Remarkable here is the
abundance of examples. Every one dimensional quasi local domain is indeed an
example of a $d$-wf-SH domain of type $1$ and so is every one dimensional
local domain with trivial ideal class group. A weakly Krull domain $D$ with
trivial $t$-class group is an example of a $t$-wf-SH domain of type $1.$ A
weakly Krull domain with zero $t$-class group is also known as a weakly
factorial domain and that, perhaps, is the reason for the abundance of
examples. Weakly factorial domains were among the earliest efforts to
generalize the notion of factoriality. These domains were initially defined
by Anderson and Mahaney in \cite{AM} as domains whose nonzero non units were
expressible as products of primary elements. Here an element $x$ of $D$ is
called primary if $xD$ is a primary ideal. Then it was shown, among other
results, in \cite{AZ}, that $D$ is a weakly factorial domain if and only if $%
D=\dbigcap\limits_{P\in X^{1}(D)}D_{P}$, the intersection is locally finite,
and $D$ has trivial $t$-class group, another way of saying that $D$ is a
weakly Krull domain with trivial $t$-class group. (At that time we did not
have the idea of christening the domains $D$ that are locally finite
intersections of localizations at height one primes as weakly Krull
domains.) To give the other properties, more important for the purposes of
that paper, it was shown that (a) $D$ is a weakly factorial domain if and
only if every convex directed subgroup of the group of divisibility of $D$
is a cardinal summand and (b) $D$ is a weakly factorial domain if and only
if the following is true: if $P$ is a prime ideal of $D$ minimal over a
proper principal ideal $xD$, then $P$ has height one and $xD_{P}$ $\cap D$
is principal. Indeed there has been a lot of activity around this concept.

On the other hand, as we come to consider the $\ast $-super homog ideals and 
$\ast $-super SH domains, things fall into the pattern of same old same old.
Just to make sure that the readers don't miss anything let's recall that the
definition of a $\ast $-super homog ideal $I$ requires that every $\ast $%
-ideal of finite type containing $I$ must be $\ast $-invertible and the
definition of a $\ast $-wf-homog ideal $I$ requires that if $I$ is $\ast $%
-invertible every $\ast $-ideal of finite type containing $I$ must be
principal. That is if $I$ is an ideal that is both $\ast $-super homog and $%
\ast $-wf-homog then every $\ast $-ideal of finite type containing $I$ is
principal. But that, in case $I$ is $\ast $-invertible, makes $I$ a $\ast $%
-f-ideal, as Definition \ref{Definition S} tells us. Conversely if $I$ is $%
\ast $-f-homog, then $I$ is obviously a $\ast $-super homog and a $\ast $%
-wf-homog ideal. This gives us the following result.

\begin{proposition}
\label{Proposition 4Ah} A $\ast $-ideal $I$ of finite type is $\ast $%
-f-homog if and only if $I$ is a $\ast $-super homog and a $\ast $-wf-homog
ideal.
\end{proposition}

We already know that a domain whose nonzero non units are products of $\ast $%
-f-elements is a GCD\ $\ast$-IRKT. All that remains is making links with
other related concepts.

\begin{proposition}
\label{Proposition 4Aj} For a domain $D$ the following statements are
equivalent: (1) $D$ is a $\ast$-IRKT whose $\ast$-super homog ideals are
also $\ast$-wf-homog, (2) $D$ is a $\ast$-wf-SH domain whose $\ast$-wf-homog
ideals are also $\ast$-super homog, (3) $D$ is a GCD $\ast$-IRKT, (4) $D$ is
a $\ast $-IRKT with $Cl_{\ast}(D)=0$, (5) $D$ is a locally GCD $\ast$-IRKT
and $Cl_{d}(D)=0.$
\end{proposition}

\begin{proof}
(1) $\Leftrightarrow $ (3) By Proposition \ref{Proposition 4Ah} and Theorem %
\ref{Theorem S2}, (1) $\Rightarrow $ (2) A $\ast $-IRKT is a $\ast $-SH
domain. Now apply Proposition \ref{Proposition 4Ah} (2) $\Rightarrow $ (4) A 
$\ast $-wf-SH domain has trivial $\ast $-class group by Proposition \ref%
{Proposition 4Ae} and every $\ast $-wf-homog ideal being $\ast $-super homog
makes $D$ a $\ast $-IRKT, (4) $\Rightarrow $ (3) Note that if $D$ is a $\ast 
$-IRKT with $Cl_{\ast }(D)=0$, then every $\ast $-invertible $\ast $-ideal
of $D$ is principal and so every $\ast $-super homog ideal of $D$ is
principal. Thus every $\ast $-super homog ideal of $D$ is $\ast $-f-homog.
Now apply Theorem \ref{Theorem S2}, (3) $\Rightarrow $ (5) $D$ being a GCD
domain implies that $D$ is locally GCD and $Cl_{t}(D)=0$ and we know that $%
Cl_{t}(D)\supseteq Cl_{d}(D)$, (5) $\Rightarrow $ (3) Let's prove the
following result.
\end{proof}

\begin{lemma}
\label{Lemma 4Ak} Let $D$ be a locally GCD domain and let $\mathcal{F}$ be a
family of nonzero primes of $D$ such that $D=\dbigcap \limits_{P\in\mathcal{F%
}}D_{P}$ is locally finite. If $Cl_{d}(D)=0$ then $D$ is a GCD domain.
\end{lemma}

\begin{proof}
Let $\tau $ be the star operation induced by $\{D_{P}\}_{P\in \mathcal{F}}.$
Since $D$ is locally GCD, each of $D_{P}$ is a GCD domain and so for each
pair $a,b$ of non zero elements of $D$ we have $(aD\cap bD)D_{P}=aD_{P}\cap
bD_{P}$ principal. Because $D=\dbigcap\limits_{P\in \mathcal{F}}D_{P}$ is
locally finite, $aD\cap bD$ is contained in at most a finite number $%
P_{1},P_{2},...,P_{n}$ of members of $\mathcal{F}$, precisely ones that
contain at least one of $a,b.$ That is $(aD\cap bD)D_{Q}=D_{Q}$ for all $%
Q\in \mathcal{F}$ such that $ab\notin Q.$ Then $(aD\cap
bD)D_{P_{i}}=x_{i}D_{P_{i}}$ where we can take $x_{i}\in P_{i}$ for some $i$
and indeed we can take $x_{i}\in aD\cap bD$, $i=1,2,...n.$ Set $%
A=(ab,x_{1},...x_{n}).$ Then $A_{v}\subseteq aD\cap bD$ because $A\subseteq
aD\cap bD$ which is a $v$-ideal. Now $(aD\cap
bD)D_{P_{i}}=x_{i}D_{P_{i}}\subseteq AD_{P_{i}}$ for $i=1,...,n$ and $%
(aD\cap bD)D_{Q}=D_{Q}=AD_{Q}$ for all $Q\in \mathcal{F}$ such that $%
ab\notin Q.$ So $(aD\cap bD)D_{P}\subseteq AD_{P}$ for all $P\in \mathcal{F}$%
. Thus $(aD\cap bD)\subseteq A_{t}\subseteq A_{v}$ and this shows that $%
(aD\cap bD)$ is a $v$-ideal of finite type. Now as $D$ is locally GCD $%
(aD\cap bD)$ is locally principal and hence flat. But a flat ideal that is
also a $v$-ideal of finite type is invertible \cite[Proposition 1]{Z90}. Now
for each pair $a,b\in D\backslash \{0\},$ $(aD\cap bD)$ is invertible and $%
Cl_{d}(D)=0$ means every invertible ideal of $D$ is principal. Whence for
each pair $a,b\in D\backslash \{0\}$ $(aD\cap bD)$ is principal and $D$ is a
GCD domain.
\end{proof}

Next from Proposition \ref{Proposition 4Ah}, we conclude that a $\ast $%
-f-homog ideal of type $1$ is nothing but a $\ast $-super homog ideal of
type $1$ that is also a $\ast $-wf-homog ideal. Again we know that a domain
whose nonzero non units are expressible as products of $\ast $-f- elements
of type $1$ is a GCD-$\ast $-GKD (cf Theorem \ref{Theorem S3}) and that
these domains were studied in \cite{AAZ1} as GUFDs with a totally different
set of definitions. We also know that only two values of $\ast ,d$ and $t$,
have any effect. That is a GUFD $D$ is a one dimensional Bezout domain if $%
\ast =d$ and a GCD-GKD if $\ast =t.$

\bigskip Let's call a $\ast $-homog ideal $I$ a weak $\ast $-almost
factorial homog ($\ast $-waf-homog) ideal if whenever $I$ is a $\ast $%
-invertible $\ast $-ideal for every $\ast $-invertible $\ast $-ideal $J$
that contains $I$ there is a positive integer $j$ such that $(J^{j})^{\ast }$
is principal.

It is easy to see that the product $IJ$ of two similar $\ast $-waf-homog
ideals $I,J$ is a $\ast $-waf-homog ideal similar to both $I$ and $J.$ ($I$
and $J$ are $\ast $-waf-homog, so $(IJ)^{\ast }$ is $\ast $-homog. Next, if $%
IJ$ is $\ast $-invertible then both $I$ and $J$ are $\ast $-invertible and
so $\ast $-af-homog, by the remark after Definition \ref{Definition 4Ac}
making $(IJ)^{\ast }$ a $\ast $-waf-homog ideal.) Consequently a product of
finitely many $\ast $-waf-homog ideals is a expressible, uniquely, up to
order, as a product of mutually $\ast $-comaximal $\ast $-waf-homog ideals.

Let's start with a clone of Corollary \ref{Corollary 4Ab}.

\begin{proposition}
\label{Proposition 4Ba} Let $D$ be a $\ast$-SH domain. Then the $\ast$-class
group of $D$ is torsion if and only if for every $\ast$-invertible $\ast $%
-homog ideal $I$ of $D$ we have $(I^{r})^{\ast}$ principal for some $r.$
\end{proposition}

Next we have a clone of the definition of weakly factorial domains.

\begin{definition}
\label{Definition 4Bb} Call an integral domain $D$ a $\ast $-SH weakly
almost factorial domain ($\ast $-waf-SH) if every nonzero non unit of $D$ is
expressible as a finite $\ast $-product of $\ast $-waf-homog ideals.
\end{definition}

\begin{proposition}
\label{Proposition 4Bc}A $\ast$-SH domain with torsion $\ast$-class group is
a $\ast$-waf-SH domain. Conversely a $\ast$-waf-SH domain is a $\ast$-SH
domain with torsion $\ast$-class group.
\end{proposition}

\begin{proof}
Indeed if the $\ast $-class group of $D$ is torsion, for every $\ast $%
-invertible $\ast $-ideal $I$ of $D$ there is a positive integer $n$ such
that $(I^{n})^{\ast }$ is principal. Now for every nonzero non unit $x$ in a 
$\ast $-SH domain $xD=(I_{1}I_{2}...I_{n})^{\ast },$ where $I_{i}$ are
mutually $\ast $-comaximal $\ast $-invertible $\ast $-ideals. With the added
restriction of torsion $\ast $-class group, each of $I_{i}$ is a $\ast $%
-invertible $\ast $-ideal which fits the definition of a $\ast $-waf-homog
ideal, i.e., $I_{i}$ is such that for every $\ast $-invertible $\ast $-ideal
say $J_{i}$ there is a positive integer $n_{i}$ such that $(J_{i})^{n_{i}}$
is principal, provided amply by the fact that the $\ast $-class group of $D$
is torsion. Conversely, it is obvious that (a) $D$ is of finite $\ast $%
-character and (b) every prime ideal of $D$ contains a $\ast $-waf-homog
ideal. It is now easy to show that every maximal $\ast $-ideal of $D$ is
spawned by a $\ast $-waf-homog ideal, because a $\ast $-waf-homog ideal is a 
$\ast $-homog ideal to start with, and that no two distinct maximal $\ast $%
-ideals contain a nonzero prime ideal. So, a $\ast $-waf-SH domain $D$ is a $%
\ast $-SH domain. Next to show that the $\ast $-class group of $D$ is
torsion, take an integral $\ast $-invertible $\ast $-ideal $I$ of $D$ and
let $0\neq x\in I.$ By Theorem \ref{Theorem 4Aa}, $I=(I_{1}I_{2}...I_{n})^{%
\ast }$ where each of $I_{i}$ is a $\ast $-homog ideal and $I_{i}$ mutually $%
\ast $-comaximal. Pick one, say $I_{k},$ and note that $%
xD=(K_{1}K_{2}...K_{r})^{\ast }\in I_{k}$ where $K_{i}$ are, mutually $\ast $%
--comaximal $\ast $-waf-homog ideals$.$ Since for each $i,K_{i}$ is $\ast $%
-homog and since $K_{i}$ are mutually $\ast $-comaximal, by Corollary \ref%
{Corollary H} only one of the $K_{i}$ is contained in $I_{k}.$ Now $K_{i}$
being a $\ast $-waf-homog ideal, $K_{i}$ (is $\ast $-invertible and) has the
property that for any $\ast $-invertible $\ast $-ideal $L$ containing $K_{i}$
there is a positive integer $t$ such that $(L^{t})^{\ast }$ is principal, by
our definition. Thus $(I_{k}^{m_{k}})^{\ast }$ is principal for some
positive integer $m_{k}.$ Now let $m=\func{lcm}(m_{1},m_{2},...,m_{n}),$
then $(I^{m})^{\ast }=((I_{1}^{m_{1}})^{(\frac{m}{m1})}(I_{2}^{m2})^{(\frac{m%
}{m2})}...(I_{n}^{m_{1}})^{(\frac{m}{mn})})^{\ast }$ is principal. Finally
as $I$ was arbitrary, we conclude that for every integral $\ast $-invertible 
$\ast $-ideal $J$ of $D,$ $(J^{m})^{\ast }$ is principal for some positive
integer $m$ and this, indeed, forces $Cl_{\ast }(D)$ to be torsion.
\end{proof}

Examples: (c) (When $\ast=d$ and no restriction on dimension). Let $(R,M)$
be a quasi local almost factorial domain, let $L$ be the quotient field of $%
R $ and let $X$ be an indeterminate over $L.$ Then the ring $D=R+XL[X]$ is
an $n+1$ dimensional $d$-SH domain such that $Cl_{d}(D),$ the ideal class
group of $D$ is zero but $Cl_{t}(D)$ is torsion, because all $d$-homog
ideals are either principal or of the form $f(X)(A+XL[X])$ where $A$ is an
ideal of $R$ and $((A+XL[X])^{r})_{t}=((A^{r})_{t}+XL[X]).$ Thus if for
every finitely generated ideal $A$ of $R$ we have a positive integer $r$
such that $(A^{r})_{t}$ is principal the corresponding ideals of the form $%
f(X)(A+XL[X]) $ have the same property.

Of course Example (c) cannot be used as an example of a $t$-waf-SH domain,
if $R$ is not $t$-local. For if $R$ has say maximal $t$-ideals $M$ and $N$
then $D$ has two corresponding maximal $t$-ideals $M+XL[X]$ and $N+XL[X],$
ensuring that $D$ is not a $t$-SH domain. Meaning that for $D$ to be a $t$%
-SH domain $R$ has to be a domain with a unique maximal $t$-ideal. But such
a domain will have to be $t$-local. Indeed if $R$ has a unique maximal $t$%
-ideal $N$ then every nonzero non unit of $R$ would be contained in $N$ and
that forces every maximal ideal of $R$ contained in $N.$ But $R$ being $t$%
-local means that $Cl_{t}(D)$ is zero, slightly more than torsion.

The reason is the following result.

\begin{proposition}
\label{Proposition 4Bd} Let $(D,M)$ be a $t$-local domain. Then (a) if $A$
is a $t$-invertible ideal of $D$ then $A$ is principal and (b) If $A$ is an
ideal of $D$ such that $(A^{n})_{t}=D$ for some positive integer $n,$ then $%
A $ is principal.
\end{proposition}

\begin{proof}
(a) If $A$ is $t$-invertible, then $AA^{-1}$ is not contained in any maximal 
$t$-ideal and so $AA^{-1}$ is not contained in $M.$ But then $AA^{-1}$ is
not contained in any maximal ideal of $D,$ because $M$ is the only maximal
ideal of $D.$ Hence. $AA^{-1}=D.$ But then $A$ is invertible in a quasi
local domain and hence principal, (b) if $(A^{n})_{t}=D$ then $A$ is $t$%
-invertible and so, by (a) above, is principal.
\end{proof}

\begin{remark}
\label{Remark 4Be} Part (a) of Proposition \ref{Proposition 4Bd} is \cite[%
Proposition 1.12]{ACZ} but the proof there is not quite clear. Here we have
made the necessary clarifications. Thus even if $(D,M)$ is an almost
valuation domain, i.e. for each pair of nonzero elements $x,y$ there is a
positive integer $n$ such that $x^{n}|y^{n}or$ , $y^{n}|x^{n},$ $%
Cl_{t}(D)=(0).$
\end{remark}

This gives us Example (b) all over again. That means $D=R+XL[X]$ is not of
much use in this context.

Indeed if, on the other hand, we consider the $\ast $-super homog ideals and 
$\ast $-super SH domains things fall into the realm of what we already know.
Let's recall that the definition of a $\ast $-super homog ideal $I$ requires
that every $\ast $-ideal of finite type containing $I$ must be $\ast $%
-invertible and the definition of a $\ast $-waf-homog ideal $I$ requires
that if $I$ is $\ast $-invertible then for every $\ast $-ideal of finite
type $J$ containing $I$ there is a positive integer $n$ such that $%
(J^{n})^{\ast }$ is principal. That is if $I$ is an ideal that is both $\ast 
$-super homog and $\ast $-waf-homog then for every $\ast $-ideal $J$ of
finite type containing $I$ there is a positive integer $n$ such that $%
(J^{n})^{\ast }$ is principal. But that, in case $I$ is $\ast $-invertible,
makes $I$ a $\ast $-af-ideal as Definition \ref{Definition S} tells us and
if we are considering a $\ast $-super homog ideal that is also a $\ast $-waf
homog ideal then we have a $\ast $-af homog ideal. Conversely if $I$ is $%
\ast $-af-homog, then $I$ is obviously a $\ast $-super homog and a $\ast $%
-waf-homog ideal because it satisfies the \textquotedblleft if $I$ is $\ast $%
-invertible\textquotedblright , vacuously . This gives us the following
result.

\begin{proposition}
\label{Proposition 4Bf} A $\ast$-ideal $I$ of finite type is $\ast$-af-homog
if and only if $I$ is a $\ast$-super homog and a $\ast$-waf-homog ideal.
\end{proposition}

We already know that a domain whose nonzero non units are products of $\ast $%
-af-elements is an AGCD\ $\ast$-IRKT. All that remains is making links with
other related concepts.

\begin{proposition}
\label{Proposition 4Bg} For a domain $D$ the following statements are
equivalent: (1) $D$ is a $\ast$-IRKT whose $\ast$-super homog ideals are
also $\ast$-wf-homog, (2) $D$ is a $\ast$ waf-SH domain whose $\ast$%
-waf-homog ideals are also $\ast$-super homog, (3) $D$ is an AGCD $\ast$%
-IRKT, (4) $D$ is a $\ast$-IRKT with $Cl_{\ast}(D)$ torsion, (5) $D$ is a
locally AGCD $\ast $-IRKT and $Cl_{d}(D)$ is torsion.
\end{proposition}

\begin{proof}
(1) $\Leftrightarrow$ (3) By Proposition \ref{Proposition 4Bf} and Theorem %
\ref{Theorem R2}, (1) $\Rightarrow$ (2) A $\ast$-IRKT is a $\ast$-SH domain.
Now apply Proposition \ref{Proposition 4Bf} (2) $\Rightarrow$ (4) A $\ast $%
-waf-SH domain has torsion $\ast$-class group by Proposition \ref%
{Proposition 4Bc} and every $\ast$-waf-homog ideal being $\ast$-super homog
makes $D$ a $\ast$-IRKT, (4) $\Rightarrow$ (3) Note that if $D$ is $\ast$%
-IRKT with $Cl_{\ast}(D)$ torsion, then every $\ast$-ideal $K$ of $D$ is is $%
\ast$-invertible such that $(K^{n})^{\ast}$ is principal for some $n$ and so
every $\ast$-homog ideal $I$ of $D$ is $\ast$-super homog such that $%
(J^{n})^{\ast}$is principal for every $\ast$-ideal $J$ of finite type
containing $I$. Thus every $\ast$-super homog ideal of $D$ is $\ast$%
-af-homog. Now apply Theorem \ref{Theorem R2}, (3) $\Rightarrow$ (5) $D$
being an AGCD domain implies that $D$ is locally AGCD and $Cl_{t}(D$ is
torsion and $Cl_{t}(D)\supseteq Cl_{\ast}(D)$, (5) $\Rightarrow$ (3) Let's
note that $D$ is a locally AGCD domain if for each maximal ideal $M$ of $D$
we have for each pair $a,b\in D$ a positive integer $n_{M}=n_{M}(a,b)$ such
that $a^{n_{M}}D_{M}\cap b^{n_{M}}D_{M}$ is principal, then prove the
following result.
\end{proof}

\begin{lemma}
\label{Lemma 4Bh} Let $D$ be a locally AGCD domain and let $D$ be of finite $%
t$- character that is $D=\dbigcap \limits_{P\in\mathcal{F}}D_{P}$ is locally
finite. If $Cl_{d}(D)$ is torsion then $D$ is an AGCD domain.
\end{lemma}

\begin{proof}
Let $w$ be the usual star operation induced by $\{D_{P}\}_{P\in t\text{-}%
\max (D)}.$Since $D$ is locally AGCD, each of $D_{P}$ is an AGCD domain and
so for each pair $a,b$ of non zero elements of $D$ we have for some positive
integer $n_{P},$ $(a^{n_{P}}D\cap b^{n_{P}}D)D_{P}=a^{n_{P}}D_{P}\cap
b^{n_{P}}D_{P}$ principal (actually, as $D_{P}$ is $t$-local and AGCD $D_{P}$
is an almost Bezout domain.) Now for each $P$ there would be a medley of
numbers ($\{n_{M(P)}$ for each maximal ideal $M$ containing $P\}$ but
choosing any one would serve our purpose. Because $D=\dbigcap\limits_{P\in t%
\text{-}\max (D)}D_{P}$ is locally finite, $aD\cap bD$ is contained in at
most a finite number $P_{1},P_{2},...,P_{r}$ of members of $t$-$\max (D)$,
ones that contain $a$ or $b.$ Choose $n=\func{lcm}%
(n_{P_{1}},n_{P_{2}},...,n_{P_{r}}).$ Now as $(a^{n_{P_{i}}}D\cap
b^{n_{P_{i}}}D)D_{Pi}=a^{n_{P_{i}}}D_{P_{i}}\cap b^{n_{Pi}}D_{P_{i}}$ is
principal, for each $i,$ and as $n_{P_{i}}|n$ we have $%
(a^{n_{P_{i}}}D_{P_{i}}\cap
b^{n_{Pi}}D_{P_{i}})^{n/n_{p_{i}}}=((a^{n_{P_{i}}})^{n/n_{p_{i}}}D_{P_{i}}%
\cap (b^{n_{Pi}})^{n/n_{p_{i}}}D_{P_{i}})=a^{n}D_{P_{i}}\cap b^{n}D_{P_{i}}$ 
$=$ $(a^{n}D\cap b^{n}D)D_{P_{i}}=d_{i}D_{P_{i}}$ principal and of course
for all those $Q\in \mathcal{F}$ such that none of $a,b$ belong to $Q$ we
have $(a^{n}D\cap b^{n}D)D_{Q}=D_{Q}$ and hence principal we conclude, as in
the proof of Lemma \ref{Lemma 4Ak} that $(a^{n}D\cap
b^{n}D)=(ab,d_{1},...,d_{r})_{w}$ $=(ab,d_{1},...,d_{r})_{v}.$ Going back
again and applying the result that if $A$ is a $t$-ideal of finite type and $%
t$-locally principal then $A$ is $t$-invertible. Now $(a^{n}D\cap b^{n}D)$
is $t$-invertible and so, of finite type and as $D$ is locally AGCD there is
for each maximal ideal $M$ a positive integer $m_{M}$ such that $%
((a^{n}D\cap b^{n}D)D_{M})^{m_{M}})_{v}=d_{M}D_{M}$ or, as $(a^{n}D\cap
b^{n}D)$ is $t$-invertible $(a^{n(m_{M})}D\cap b^{n(m_{M})}D)=d_{M}D_{M}.$
Thus by Theorem 2.3 of \cite{A} there is a positive integer $m$ such that $%
((a^{n}D\cap b^{n}D)^{m})_{v}=(a^{nm}D\cap b^{nm}D)$ is invertible. But as
the $d$-class group of $D$ is torsion we have $(a^{nm}D\cap b^{nm}D)^{r}=dD$
for a positive integer $r$ and for $d\in D.$ Proving that for for each pair $%
a,b\in D\backslash 0$ there is a positive integer $t$ such that $a^{t}D\cap
b^{t}D$ is principal.
\end{proof}

The above results can give us more examples of general $\ast $-waf-SH
domains and indeed it may not be too hard to construct examples of $\ast $%
-waf-SH domains of higher dimensions. But, as it stands, most of the
available examples are one dimensional. So, for now, we look at one
dimensional $\ast $-waf-SH domains. For that let's start with the definition
of $\ast $-waf-homog ideals. We can say that a $\ast $-homog ideal of type $%
1 $ that is a $\ast $-waf ideal as well is a $\ast $-waf-homog ideal of type 
$1.$ Similarly we can just breeze through other definitions and results
saying that a $\ast $-waf-SH domain of type $1.$ In the $t$-dimension $1$
scenario one source that stands out is \cite{AMo}. In it, Anderson and Mott
discuss domains with only finitely many non-associated irreducible elements.
These domains are called Cohen Kaplansky domains, because Cohen and
Kaplansky were the first to study them in \cite{CK}. It turns out that
CK-domains are weakly Krull domains with only a finite number of maximal $t$%
-ideals and $D_{P}$ is a CK-domain for each maximal $t$-ideal $P.$ Indeed
each maximal $t$-ideal is of height one and maximal, this is because of the
fact that if $D$ has only a finite number of maximal $t$-ideals then these
maximal $t$-ideals are precisely the maximal ideals of $D.$ It was also
established in \cite{AMo} that a CK-domain $D$ is an AGCD domain that
happens to have $Cl_{t}(D)=0.$ In other words a CK-domain is weakly
factorial domain and an almost weakly factorial domain.

The other important source of examples is \cite{AnCP}. In this paper the
authors study under the name of generalized weakly factorial domains the
domain whose nonzero non units $x$ have the property that for each $x$ there
is a positive integer $n$ such that $x^{n}$ is a product of primary
elements. These are weakly Krull domains with torsion $t$-class group.
(Indeed as $xD=((xD_{P_{1}}\cap D)(xD_{P_{2}}\cap D)...(xD_{P_{r}}\cap
D))_{t}$ and the $t$-class group is torsion we get the same result.)

The second author got interested in generalizing the existing notions of
unique factorization from Professor P.M. Cohn's work. Perhaps the second
author was not too interested in non-commutative algebra, that Cohn was so
admirably good at, the second author chose to concentrate on unique
factorization in commutative ring theory. His first attempt was the theory
of GUFDs. Then he tried to mimic Cohn's rigid factorizations \cite{C1} in
the commutative rings. Apparently all he had to go on was that if $r$ is
rigid in the non-commutative domain $R$ then the lattice $L(Rr,R)$ was a
chain and that Paul Cohn used 2-firs for rigid factorizations. Another good
yet brief source, if you want to have a quick idea is Cohn's survey on UFDs 
\cite{C}. Now in the commutative case, $r$ being rigid boils down to a
non-unit $r$ such that for all $x,y|r$ we have $x|y$ or $y|x.$ But then an
irreducible element is also rigid and products of irredible elements produce
unique factorization under some very stringent conditions. Now Cohn's 2-firs
in the commutative case are Bezout domains. It was easy to show that in a
Bezout domain a product of finitely many rigid elements can be uniquely
written as a product of mutually coprime rigid elements. So, he tried to see
if a product of finitely many rigid elements in a GCD domain $D$ is uniquely
expressible as a product of mutually coprime rigid elements. It worked and
he wrote his paper on Semirigid GCD domains \cite{Zr2}. But the question
was: How to define a rigid element so that in a general commutative domain $%
D $ a finite product of rigid elements is uniquely expressible as a product
of mutually coprime \textquotedblleft improved\textquotedblright\ rigid
elements? The definition of $\ast $-f-homog does that. Now the question is:
Can we do something similar to the definition of rigid in the
non-commutative case, to get better results?

\bigskip

\bigskip

\bigskip

\bigskip

\bigskip

\bigskip

\bigskip


\begin{thebibliography}{99}
\bibitem{A} D.D. Anderson, Globalization of some local properties in Krull
domains, Proc. Amer. Math. Soc. 85(1982), 141-145.

\bibitem{AAZ1} D.D. Anderson, D.F. Anderson, and M. Zafrullah, A
generalization of unique factorization, Boll. Un. Mat. Ital. A(7)9(1995),
401-413.

\bibitem{AAZ} D.D. Anderson, D.F. Anderson, and M. Zafrullah, Completely
integrally closed Prufer $v$-multiplication domains, Comm. Algebra 45(2017),
5264-5282.

\bibitem{AC} D.D. Anderson and S.J. Cook, Two star operations and their
induced lattices, Comm. Algebra 28(2000), 2461-2476.

\bibitem{ACZ} D.D. Anderson, G.W. Chang, and M. Zafrullah, Integral domains
of finite $t$-character, J. Algebra 396(2013), 169-183.

\bibitem{AM} D.D. Anderson and L. A. Mahaney, On primary factorizations, J.
Pure Appl. Algebra 54(1988), 141-154.

\bibitem{AMo} D.D. Anderson and J. Mott, Cohen-Kaplansky domains: Integral
domains with a finite number of irreducible elements, J. Algebra 148(1992),
17-41.

\bibitem{AMZ} D.D. Anderson, J. Mott, and M. Zafrullah, Finite character
representations for integral domains, Boll. Un. Mat. Ital. B(7)6(1992),
613-630.

\bibitem{AZ} D.D. Anderson and M. Zafrullah, Weakly factorial domains and
groups of divisibility, Proc. Amer. Math. Soc. 109(1990), 907-913.

\bibitem{AZ1} D.D. Anderson and M. Zafrullah, Almost Bezout domains, J.
Algebra 142(1991), 285-309.

\bibitem{AZ2} D.D. Anderson and M. Zafrullah, Independent locally-finite
intersections of localizations, Houston J. Math. 25(1999), 433-449.

\bibitem{An} D.F . Anderson, A general theory of class groups, Comm. Algebra
16(1988), 805-847.

\bibitem{AnCP} D.F. Anderson, G.W. Chang and J. Park, Generalized weakly
factorial domains, Houston J. Math, 29(2003), 1-13.

\bibitem{B} A. Bouvier, Le groupe des classes d'un anneau integre, 107 erne
Congres National des Societe Savantes, Brest , Fasc. IV(1982), 85-92.

\bibitem{BZ} A. Bouvier and M. Zafrullah, On some class groups of an
integral domain, Bull. Soc. Math. Grece 29(1988), 45-59.

\bibitem{C} Unique factorization domains, Amer. Math. Monthly, 80 (1 )(1973)
1-18.

\bibitem{C1} P.M. Cohn, Free Rings and Their Relations, Academic Press,
London 1985.

\bibitem{CK} I.S. Cohen and I. Kaplansky, Rings with a finite number of
primes I, Trans. Amer. Math. Soc. 60(1946), 468-477.

\bibitem{CMZ} D. Costa, J. Mott and M. Zafrullah, The construction $%
D+XD_{S}[X]$, J. Algebra 53(1978) 423-439.

\bibitem{DZ} T. Dumitrescu and M. Zafrullah, Characterizing domains of
finite $\ast $-character, J. Pure Appl. Algebra 214(2010), 2087-2091.

\bibitem{EGZ} S. El Baghdadi, S. Gabelli, and M. Zafrullah, Unique
representation domains II, J. Pure Appl. Algebra 212(2008), 376-393.

\bibitem{gilmer} R. Gilmer, Multiplicative Ideal Theory, Marcel Dekker, New
York, 1972.

\bibitem{Gr1} M. Griffin, Rings of Krull type, J. Reine Angew. Math.
229(1968), 1-27.

\bibitem{HMM} E. Houston, S. Malik and J. Mott, Characterizations of
*-multiplication domains, Canad. Math. Bull., 27(1)(1984) 48-52.

\bibitem{HZ2} E. Houston and M. Zafrullah, $\ast $-Super potent domains, I.
Commut. Algebra, (to appear)

\bibitem{Kap} I. Kaplansky, Commutative Rings. Revised Edition, University
of Chicago Press, Chicago and London, 1974.

\bibitem{Mat} E. Matlis, Torsion-free modules, The University of Chicago
Press, Chicago, London, 1972.

\bibitem{Olb} Bruce Olberding, Characterizations and constructions of
h-local domains, in Models, modules and abelian groups, 385-406, Walter de
Gruyter, Berlin, 2008.

\bibitem{Ribenboim} P. Ribenboim, Anneaux normaux r\'{e}els \`{a} caract\`{e}%
re fini, Summa Brasil. Math. 3(1956), 213-253.

\bibitem{S} U. STORCH, Fastfaktorielle rings, in Schriftenreihe Math. Inst.
Univ. Munster, Vol. 36, University of Munster, Munster, 1967.

\bibitem{Zd} M. Zafrullah, Generalization of unique factorization and
related topics, PhD Dissertation, The University of London, 1974.

\bibitem{Zr2} M. Zafrullah, Semirigid GCD domains, Manuscripta Math.
17(1975), 55-66.

\bibitem{ZR3} M. Zafrullah, Rigid elements in GCD domains\ J. Natur. Sci.
and Math. \textbf{17}(1977), 7-14.

\bibitem{Z} M. Zafrullah, A general theory of almost factoriality,
Manuscripta Math. 51(1985), 29-62.

\bibitem{Z87} M. Zafrullah, On a property of pre-Schreier domains, Comm.
Algebra 15(1987), no. 9, 1895-1920.

\bibitem{Z90} M. Zafrullah, Flatness and invertibility of an ideal, Comm.
Algebra, 18(1990), 2151-2158.

\bibitem{z00} M. Zafrullah, Putting t-invertibility to use, in,
Non-Noetherian Commutative Ring Theory 429-457, Math. Appl., 520, Kluwer
Acad. Publ., Dordrecht, 2000.

\bibitem{Z2010} M. Zafrullah, t-Invertibility and Bazzoni-like statements,
J. Pure Appl. Algebra 214(2010), 654--657.
\end{thebibliography}
\end{document}